\documentclass[a4paper,10pt]{article}
\usepackage{amscd}
\usepackage{amsmath}
\usepackage{bbding}
\usepackage{amsfonts}
\usepackage{cite}
\usepackage{mathrsfs}
\usepackage{amssymb}
\allowdisplaybreaks[2]
\usepackage{leftidx}
\usepackage{graphicx,latexsym,amsmath,amssymb,amsthm,enumerate}
\usepackage{float}

\usepackage[dvipdfm,
           pdfstartview=FitH,
           colorlinks=true,
             pdfborder=001,
           citecolor=blue]{hyperref}

9.3in\textwidth 6.3in \hoffset -2.0cm \DeclareFontEncoding{OT2}{}{}
\DeclareTextSymbol{\cyrsftsn}{OT2}{126}
\DeclareTextSymbol{\textnumero}{OT2}{125}

\theoremstyle{definition}
\newtheorem{theorem}{Theorem}[section]
\newtheorem{lemma}{Lemma}[section]
\newtheorem{corollary}{Corollary}[section]
\newtheorem{proposition}{Proposition}[section]
\newtheorem{definition}{Definition}[section]
\newtheorem{remark}{Remark}[section]
\newtheorem{example}{Example}[section]

\begin{document}
\title{{\LARGE\bf{On Ekeland's variational principle for interval-valued functions with applications}}\thanks{This work was supported by the National Natural Science Foundation of China (11471230, 11671282).}}
\author{ Chuang-liang Zhang $^{1,2}$\thanks{E-mail: clzhang1992@sina.com}\, and Nan-jing Huang $^{2}$\thanks{Corresponding author, E-mail: nanjinghuang@hotmail.com}
\\
{\small\it $^1$ School of Mathematics, Jiaying University, Meizhou 514015, Guangdong,  P.R. China}\\
{\small\it $^2$ Department of Mathematics, Sichuan University, Chengdu 610064, Sichuan, P.R. China}}
\date{ }
\maketitle
\begin{flushleft}
\hrulefill\\
\end{flushleft}
 {\bf Abstract}.
In this paper, we obtain a version of Ekeland's variational principle for interval-value functions by means of the Dancs-Heged\"{u}s-Medvegyev theorem \cite{DHM}. We also derive two versions of Ekeland's variational principle involving the generalized Hukuhara G\^{a}teaux differentiability of interval-valued functions as well as a version of Ekeland's variational principle for interval-valued bifunctions. Finally, we apply these new versions of Ekeland's variational principle to fixed point theorems, to interval-valued optimization problems, to the interval-valued Mountain Pass Theorem, to noncooperative interval-valued games, and to interval-valued optimal control problems described by interval-valued differential equations.
\begin{flushleft}
\hrulefill
\end{flushleft}

\section{Introduction}
\label{sec1}
The variational principle was obtained by Ekeland \cite{Ekeland1}, now known as Ekeland's variational principle, namely:

Let $(X,d)$ be a complete metric space and $f:X\to \mathbb{R}\cup\{+\infty\}$ be a lower semicontinuous function which is lower bounded and $f\not\equiv+\infty$. Then for every $\varepsilon>0$ and every $x_0\in X$ satisfying $f(x_0)\leq \inf_{x\in X}f(x)+\varepsilon$, and every $\lambda>0$, there exists $\overline{x}\in X$ such that
\begin{itemize}
\item[(i)] $f(\overline{x})\leq f(x_0)$;
\item[(ii)] $d(x_0,\overline{x})\leq\lambda$;
\item[(iii)] $\forall x\neq\overline{x}$, $f(x)>f(\overline{x})-\frac{\varepsilon}{\lambda}d(x,\overline{x}).$
\end{itemize}
It is well known that Ekeland's variational principle becomes a powerful tool in the study of many problems arising in nonlinear analysis, dynamical system, critical points theory, economics and finance, optimization and control theory (see, for instance, \cite{AZ,AE,Chang,Ekeland2,Ekeland3,GRTZ,Mo}). Thereafter, extensive efforts have been devoted into many different versions of Ekeland's variational principle with vector-valued functions or set-valued mappings and bifunctions (see, for instance, \cite{Araya,BKP,CHY,CG,GJN,GKNR,HZ,KTZ,Qiu1,ZH}).

By using the concept of the Mordukhovich coderivative, Ha \cite{Ha} obtained some versions of Ekeland's variational principle for set-valued mappings in 2005. Ha \cite{Ha1} further established two set-valued versions of Ekeland's variational principle by employing the Clarke normal cone in 2006. Over the past years, Ha \cite{Ha2}, Khanh and Quy \cite{KQ}, Liu and Ng \cite{LN}, Qiu \cite{Qiu}, and Qiu et al. \cite{QSH} generalized different versions of Ekeland's variational principle for set-valued mappings under some suitable conditions. We note that, except for \cite{Ha,Ha1,Ha2}, all these works deal with Ekeland's variational principle for set-valued mappings without derivatives, since it is difficult to calculate derivatives for set-valued mappings. As a result, applications of set-valued versions of Ekeland's variational principle have been limited to the case without derivatives. For instance, the Mountain Pass Theorem and nonlinear Euler-Lagrange equations in set-valued cases have not been established by employing the known set-valued versions of Ekeland's variational principle. It is worth mentioning that, interval analysis, regarded as a particular subject of set-valued analysis, has been widely applied to solve many real problems arising in decision science, neural computation, artificial intelligence, optimization problems, and numerical algorithms (see, for instance, \cite{Moore1,Moore2,Wu}). We also note that the Hukuhara difference between two intervals does not necessarily exist in general. To address inefficiency of this notion, Stefanini and Bede \cite{SB} introduced the generalized Hukuhara difference (for short, $gH$-difference) and generalized Hukuhara differentiability (for short, $gH$-differentiability) to study interval-valued/fuzzy differential equations. Since then, the $gH$-difference and $gH$-differentiability have been widely used in the study of interval analysis and fuzzy set theory (see, for example, \cite{BS,CMSJ,LLSM,OCHR,OHCR,QD}). Recently,  by using the concept of $gH$-difference, Ghosh et al. \cite{GCMD} introduced the generalized Hukuhara G\^{a}teaux differentiability (for short, $gH$-G\^{a}teaux differentiability) and generalized Hukuhara Fr\'{e}chet differentiability (for short, $gH$-Fr\'{e}chet differentiability) to investigate interval-valued optimization problems.

Usually, the $gH$-G\^{a}teaux differentiability of interval-valued functions can be easily calculated and checked. The overarching goal of this paper is to establish some new versions of Ekeland's variational principle for interval-valued functions by using the $gH$-G\^{a}teaux differentiability, which have not yet been obtained in the previous literature.  According to the order relation defined on intervals, we first introduce the concepts of lower semicontinuity and lower boundedness for interval-valued functions. We then establish an interval-valued version of Ekeland's variational principle by employing the Dancs-Heged\"{u}s-Medvegyev theorem. Moreover, we deduce two interval-valued versions of Ekeland's variational principle associated with the $gH$-G\^{a}teaux differentiability as well as a new version of Ekeland's variational principle for interval-valued bifunctions. Finally, the established results are applied to obtain Carist's fixed point theorems for set-valued and interval-valued mappings, to derive some existence theorems for minimal solutions of interval-valued optimization problems, to prove the Mountain Pass Theorem for interval-valued functions, to show the existence of approximate Nash equilibria for a noncooperative interval-valued game, as well as to establish the existence of approximate minimal solutions for interval-valued optimal control problems governed by interval-valued differential equations under the generalized differentiability.

This paper is organized in the following way. The next section provides some basic notions, notations and desired results.  After that in Section 3, we show some new results concerned with Ekeland's variational principle for interval-valued functions. Before we summarize this paper in Section 5, some applications of new Ekeland's variational principles for interval-valued functions are given in Section 4.
\section{Preliminaries}\noindent
\label{sec2}
\setcounter{equation}{0}

Let $(X,d)$ be a complete metric space and $\mathcal{I}$ denote the set of all closed and bounded intervals of $\mathbb{R}$, that is, $\mathcal{I}=\{[a,b]:a\leq b,a,b\in\mathbb{R}\}$.
Given $A=[\underline{a},\overline{a}],B=[\underline{b},\overline{b}]\in \mathcal{I}$ and $\lambda\in\mathbb{R}$, the addition and scalar multiplication in $\mathcal{I}$ are given by
\begin{equation}
A+B=[\underline{a}+\underline{b},\overline{a}+\overline{b}]\nonumber\\
\end{equation}
and
\begin{equation}
\lambda A=
\begin{cases}
[\lambda \underline{a},\lambda \overline{a}] & \text{if $\lambda\geq0,$}\\
[\lambda \overline{a},\lambda \underline{a}] &\text{if $\lambda<0.$}\nonumber
\end{cases}
\end{equation}
Obviously, $A-B:=A+(-B)=[\underline{a}-\overline{b},\overline{a}-\underline{b}]$.

For two intervals $A,B\in\mathcal{I}$, the Hausdorff distance $d_H$ between $A$ and $B$ is defined by
$$
d_H(A,B)=\max\{\sup_{a\in A}d(a,B),\sup_{b\in B}d(b,A)\},
$$
where $d(a,B):=\inf_{b\in B}|a-b|$.
Equivalently, if $A=[\underline{a},\overline{a}]$ and $B=[\underline{b},\overline{b}]$, then the Hausdorff distance can be defined by
$$
d_H(A,B)=\max\{|\underline{a}-\underline{b}|,|\overline{a}-\overline{b}|\}.
$$
Clearly, $(\mathcal{I},d_H)$ is a complete metric space. It is easy to check that the following statements are true:
(i) for any $A,B,C\in \mathcal{I}$, $d_H(A+C,B+C)=d_H(A,B)$; (ii) for any $A,B\in \mathcal{I}$ and $\lambda\in\mathbb{R}$, $d_H(\lambda A,\lambda B)=|\lambda|d_H(A,B)$; (iii) for any $A,B,C,D\in \mathcal{I}$, $d_H(A+C,B+D)\leq d_H(A,B)+d_H(C,D)$.

We say that a sequence $\{A_n\}\subset\mathcal{I}$ converges to $A\in\mathcal{I}$ in the Hausdorff sense, denoted by
$$
\lim_{n\rightarrow+\infty}A_n=A\; (A_n\rightarrow A)
$$
if and only if $d_H(A_n,A)\rightarrow0$. If $A_n=[\underline{a}_n,\overline{a}_n]$ and $A=[\underline{a},\overline{a}]$ in $\mathcal{I}$, then
$d_H(A_n,A)\rightarrow 0$ if and only if $|\underline{a}_n-\underline{a}|\rightarrow0$ and $|\overline{a}_n-\overline{a}|\rightarrow0$.

For any $A,B\in \mathcal{I}$, if there exists $C\in \mathcal{I}$ such that $A=B+C$, then $C$ is called the Hukuhara difference between $A$ and
$B$. We denote it by $C=A\ominus_H B$. In general, $A\ominus_H B$ does not necessarily exist and so Stefanini and Bede \cite{BS} introduced the following concept.
\begin{definition}\label{de2.1}
Let $A=[\underline{a},\overline{a}]$ and $B=[\underline{b},\overline{b}]$ in $\mathcal{I}$. The $gH$-difference between $A$ and $B$ is defined by the interval $C\in\mathcal{I}$ such that
$$
A\ominus_{gH}B=C\Leftrightarrow C=A\ominus_H B\; \textrm{or}\; -C=B\ominus_H A.
$$
\end{definition}
\begin{remark}\label{re2.1}
It follows from \cite{BS,SB} that (i) $A\ominus_{gH}B=[\min\{\underline{a}-\underline{b},\overline{a}-\overline{b}\},\max\{\underline{a}-\underline{b},\overline{a}-\overline{b}\}]$; (ii)
$-(A\ominus_{gH}B)=B\ominus_{gH}A$; (iii) $A\ominus_{gH}A=\{0\}=[0,0]$.
\end{remark}

\begin{remark}\label{re2.2}
Obviously, $\lim_{n\rightarrow+\infty}(A_n\ominus_{gH}A)=[0,0]$ if and only if $\lim_{n\rightarrow+\infty}A_n=A$.
\end{remark}

\begin{definition}\label{de2.2}
For two intervals $A=[\underline{a},\overline{a}]$ and $B=[\underline{b},\overline{b}]\in\mathcal{I}$, we say that
\begin{itemize}
\item[(i)] $A\preccurlyeq B$ if $\underline{a}\leq \underline{b}$ and $\overline{a}\leq\overline{b}$;
\item[(ii)] $A\prec B$ if $\underline{a}< \underline{b}$ and $\overline{a}<\overline{b}$;
\item[(iii)] $A\not\preccurlyeq B$ if at leat one of $\underline{a}\leq \underline{b}$ and $\overline{a}\leq\overline{b}$ is not true;
\item[(iv)] $A\not\prec B$ if at least one of $\underline{a}< \underline{b}$ and $\overline{a}<\overline{b}$ is not true.
\end{itemize}
\end{definition}
\begin{remark}\label{re2.3}
Clearly, $\preccurlyeq$ is a partial order on $\mathcal{I}$, namely, it satisfies the following conditions: (i) for any $A\in\mathcal{I}$, $A\preccurlyeq A$; (ii) for any $A,B,C\in\mathcal{I}$, $A\preccurlyeq B$ and $B\preccurlyeq C$ imply that $A\preccurlyeq C$; (iii) for any $A,B\in\mathcal{I}$, $A\preccurlyeq B$ and $B\preccurlyeq A$ imply that $A=B$.
\end{remark}
\begin{remark}\label{re2.4}
Obviously, (i) $A\preccurlyeq B$ $\Leftrightarrow$ $A\ominus_{gH}B\preccurlyeq[0,0]$; (ii) $A\prec B$ $\Leftrightarrow$ $A\ominus_{gH}B\prec[0,0]$; (iii) $A\prec B\Rightarrow A\preccurlyeq B$; (iv) if $A\not\preccurlyeq B$, then $A\not\prec B$.
\end{remark}

We note that the converse of (iv) in Remark \ref{re2.4} is not true.  For instance, given two intervals $A=[1,3]$ and $B=[2,3]$, we have $A\not\prec B$. However, $A\not\preccurlyeq B$ is not true, that is, $A\preccurlyeq B$.

\begin{proposition}\label{p2.1}
Let $A,B,C,D\in\mathcal{I}$. Then the following statements are true:
\begin{itemize}
\item[(i)] If $A\preccurlyeq B$ and $C\preccurlyeq D$, then $A+C\preccurlyeq B+D$;
\item[(ii)] If $A\preccurlyeq B$, then $\lambda A\preccurlyeq \lambda B$ for all $\lambda\geq0$;
\item[(iii)] If $A+B\preccurlyeq C$, then $A\ominus_{gH}C\preccurlyeq-B$. The converse is true if $B$ is a singleton.
\end{itemize}
\end{proposition}
\begin{proof}
From the definition of $\preccurlyeq$, we can get (i) and (ii). Thus, we only need to prove (iii). To this end, let $A=[\underline{a},\overline{a}]$, $B=[\underline{b},\overline{b}]$ and $C=[\underline{c},\overline{c}]$ be such that $A+B\preccurlyeq C$.
Then it follows that $\underline{a}-\underline{c}\leq-\underline{b}$ and $\overline{a}-\overline{c}\leq-\overline{b}$. By Remark \ref{re2.1}, we know that $A\ominus_{gH}C=[\min\{\underline{a}-\underline{c},\overline{a}-\overline{c}\},\max\{\underline{a}-\underline{c},\overline{a}-\overline{c}\}]$, which leads to the following two cases.

Case 1. If $\underline{a}-\underline{c}\leq\overline{a}-\overline{c}$, then $\underline{a}-\underline{c}\leq\overline{a}-\overline{c}\leq-\overline{b}\leq-\underline{b}$ and so
$A\ominus_{gH}C=[\underline{a}-\underline{c},\overline{a}-\overline{c}]\preccurlyeq[-\overline{b},-\underline{b}]=-B$.

Case 2. If $\underline{a}-\underline{c}>\overline{a}-\overline{c}$, then
$A\ominus_{gH}C=[\overline{a}-\overline{c},\underline{a}-\underline{c}]\preccurlyeq[-\overline{b},-\underline{b}]=-B$
and so $A\ominus_{gH}C\preccurlyeq-B$.

Conversely, let $A=[\underline{a},\overline{a}]$, $B=\{b\}=[b,b]$ and $C=[\underline{c},\overline{c}]$ be such that $A\ominus_{gH} C\preccurlyeq-B$. Then by $A\ominus_{gH}C=[\min\{\underline{a}-\underline{c},\overline{a}-\overline{c}\},\max\{\underline{a}-\underline{c},\overline{a}-\overline{c}\}]$,
we can see that $\underline{a}-\underline{c}\leq -b$ and $\overline{a}-\overline{c}\leq-b$. Thus, $\underline{a}+b\leq\underline{c}$ and $\overline{a}+b\leq\overline{c}$, and so $A+B\preccurlyeq C$.
\end{proof}

The following example shows that the converse of Proposition \ref{p2.1} (iii) is not true if $B$ is not a singleton.
\begin{example}\label{ex2.1}
Let $A=[1,3]$, $B=[-3,0]$ and $C=[1,2]$. Clearly, $A\ominus_{gH}C=[0,1]$ and $-B=[0,3]$. Then we can see that
 $A\ominus_{gH}C\preccurlyeq -B$. However, $A+B=[-2,3]\not\preccurlyeq [1,2]=C$.
\end{example}

\begin{remark}\label{re2.5}
If the partial order $\preccurlyeq$ in Proposition \ref{p2.1} is replaced by $\prec$, then
\begin{itemize}
\item[(a)] If $A\prec B$ and $C\prec D$, then $A+C\prec B+D$;
\item[(b)] If $A\prec B$, then $\lambda A\prec \lambda B$ for all $\lambda>0$;
\item[(c)] If $A+B\prec C$, then $A\ominus_{gH}C\prec-B$. The converse is true if $B$ is a singleton.
\end{itemize}
\end{remark}

\begin{proposition}\label{p2.2}
Let $A_n=[\underline{a}_n,\overline{a}_n]$, $A=[\underline{a},\overline{a}]$ and $B=[\underline{b},\overline{b}]$.  If $A_n\not\prec B$ for $n=1,2,\cdots$, and $d_H(A_n,A)\rightarrow0$, then $A\not\prec B$.
\end{proposition}
\begin{proof}
Suppose to the contrary that $A\prec B$. Then $\underline{a}<\underline{b}$ and $\overline{a}<\overline{b}$. By $d_H(A_n,A)\rightarrow0$, one has
$|\underline{a}_n-\underline{a}|\rightarrow0$ and $|\overline{a}_n-\overline{a}|\rightarrow0$. Thus, there exists $N_1>0$ such that, for each $n>N_1$, $\underline{a}_n<\underline{b}$, and exists $N_2>0$ such that, for each $n>N_2$, $\overline{a}_n<\overline{b}$. Taking $N=\max\{N_1,N_2\}$, for each $n>N$, we have $\underline{a}_n<\underline{b}$ and $\overline{a}_n<\overline{b}$, and so $A_n\prec B$, which is a contradiction.
\end{proof}

We give an example to illustrate Proposition \ref{p2.2}.
\begin{example}\label{ex2.2}
Let $A_n=[\frac{1}{n},1]$ for $n=1,2,\cdots$, $A=[0,1]$ and $B=[0,2]$. Then we can see that $A_n\not\prec B$. Moreover, we have $\lim_{n\rightarrow+\infty} A_n=[0,1]=A\not\prec B$.
\end{example}
The following example shows that Proposition \ref{p2.2} is not true if $\not\prec$ is replaced by $\not\preccurlyeq$.
\begin{example}\label{ex2.3}
Let $A_n=[\frac{1}{n},2]$ for $n=1,2,\cdots$, $A=[0,2]$ and $B=[0,3]$. Then we can see that $A_n\not\preccurlyeq B$. On the other hand, we have $\lim_{n\rightarrow+\infty} A_n=[0,2]=A$. However, it is easy to see that $A=[0,2]\preccurlyeq [0,3]=B$.
\end{example}

Now we turn to consider interval-valued functions.  Let $f:X\to\mathcal{I}$ be an interval-valued function with the form $f(x)=[\underline{f}(x),\overline{f}(x)]$  for all $x\in X$, where $\underline{f}$ and $\overline{f}$ are both real-valued functions from $X$ to $\mathbb{R}$, satisfying $\underline{f}(x)\leq\overline{f}(x)$.
\begin{definition}\label{de2.3}(\cite{GCMD})
Let $X_0$ be a nonempty subset of $X$ and $f:X_0\to\mathcal{I}$ be an interval-valued function. We say that $f$ is $\preccurlyeq$-lower bounded on $X_0$ if there exists an interval $A\in\mathcal{I}$ such that, for any $x\in X_0$, $A\preccurlyeq f(x)$. Similarly, we can define $\preccurlyeq$-upper boundedness of $f$.  We call that $f$ is $\preccurlyeq$-bounded on $X_0$ if $f$ is both $\preccurlyeq$-lower and $\preccurlyeq$-upper bounded on $X_0$.
\end{definition}

\begin{remark}\label{re2.6}
Clearly, if $f:X_0\to \mathcal{I}$ is an interval-valued function, then $f$ is $\preccurlyeq$-lower (resp.,-upper) bounded if and only if $\underline{f}$ and $\overline{f}$ are both lower (resp., upper) bounded.
\end{remark}
Let $f: X_0\to \mathcal{I}$ be an interval-valued function. From \cite{WW1,WW2}, an interval $M\in\mathcal{I}$ is called the infimum of $f$ if  $M\preccurlyeq f(x)$ for all $x\in X_0$ and for each $M_1\in\mathcal{I}$, $M_1\preccurlyeq M$ whenever $M_1\preccurlyeq f(x)$ for all $x\in X_0$. We write $M=\textrm{Inf}_{x\in X_0} f(x)$. We note that $\textrm{Inf}_{x\in X_0} f(x)=[\inf_{x\in X_0} \underline{f}(x),\inf_{x\in X_0} \overline{f}(x)]$.  Also, if the infimum $\textrm{Inf}_{x\in X_0} f(x)$ exists (i.e., $\inf_{x\in X_0} \overline{f}(x)\geq\inf_{x\in X_0}\underline{f}(x)>-\infty$), then for any given $\varepsilon>0$, there exists a point $x'\in X_0$ such that $f(x')\preccurlyeq \textrm{Inf}_{x\in X_0} f(x)+[\varepsilon,\varepsilon]$. Similarly, we can define the supremum of $f$ and we denote by $\textrm{Sup} _{x\in X_0} f(x)$. Moreover, we know that infimum $\textrm{Inf}_{x\in X_0} f(x)$ (resp., supremum $\textrm{Sup} _{x\in X_0} f(x)$) exists if $f$ has $\preccurlyeq$-lower (resp., $\preccurlyeq$-upper ) bounded on $X_0$.

Now we recall the semicontinuity of functions as follows:
\begin{definition}\label{de2.4}
A function $f:X\to \mathbb{R}$ is said to be lower semicontinuous at $x_0\in X$ if,  for any sequence $\{x_n\}$ in $X$ converges to $x_0$, one has
$f(x_0)\leq\liminf_{n\rightarrow+\infty} f(x_n).$
A function $f:X\to \mathbb{R}$ is said to be lower semicontinuous on $X$ if $f$ is lower semicontinuous at each $x_0\in X$.
\end{definition}

It is well known that a function $f:X\to\mathbb{R}$ is lower semicontinuous on $X$ if and only if, for any $r\in\mathbb{R}$, the set $\{x\in X: f(x)\leq r\}$ is closed.

This fact motives us to introduce the semicontinuity of interval-valued functions.
\begin{definition}\label{de2.5}
An interval-valued function $f:X\to\mathcal{I}$ is said to be $\preccurlyeq$-lower semicontinuous on $X$ if, for any interval $A\in\mathcal{I}$, the set $\{x\in X:f(x)\preccurlyeq A\}$ is closed.
\end{definition}

\begin{proposition}\label{p2.3}
An interval-valued function $f:X\to\mathcal{I}$ is $\preccurlyeq$-lower semicontinuous on $X$ if and only if $\underline{f}$ and $\overline{f}$ are both lower semicontinuous on $X$.
\end{proposition}
\begin{proof}
($\Leftarrow$) Let $\underline{f}$ and $\overline{f}$ be lower semicontinuous on $X$. For any $A=[\underline{a},\overline{a}]\in\mathcal{I}$, taking any sequence $\{x_n\}$ in $\{x\in X:f(x)\preccurlyeq A\}$ such that $d(x_n, x_0)\rightarrow0$, we have $\underline{f}(x_n)\leq \underline{a}$ and $\overline{f}(x_n)\leq\overline{a}$. Because $\underline{f}$ and $\overline{f}$ are lower semicontinuous on $X$, we get $\underline{f}(x_0)\leq\underline{a}$ and $\overline{f}(x_0)\leq\overline{a}$, which imply that $f(x_0)\preccurlyeq A$, and so $f$ is $\preccurlyeq$-lower semicontinuous on $X$.

($\Rightarrow$) Let $f$ be $\preccurlyeq$-lower semicontinuous on $X$. For any $r\in \mathbb{R}$, we take any sequence $\{x_n\}$ in $\{x\in X:\overline{f}(x)\leq r\}$ such that $d(x_n, x_0)\rightarrow0$. Then we get $[\underline{f}(x_n),\overline{f}(x_n)]\preccurlyeq [r,r]$. By the $\preccurlyeq$-lower semicontinuity of $f$, we know that $\overline{f}(x_0)\leq r$, which says that $\overline{f}$ is lower semicontinuous on $X$.
Next, let $s\in\mathbb{R}$ and $\{y_n\}$ be a sequence in $\{x\in X:\underline{f}(x)\leq s\}$ such that $d(y_n,y_0)\rightarrow 0$. Because $\overline{f}$ is lower semicontinuous and $\{y_n\}\cup\{y_0\}$ is compact, there exists $s_1\in\mathbb{R}$ and a subsequence $\{y_{n_k}\}$ of $\{y_n\}$ satisfying $\overline{f}(y_{n_k})\leq s_1$. Thus, we can choose $\eta\geq\max\{s,s_1\}$ and so $[\underline{f}(y_{n_k}),\overline{f}(y_{n_k})]\preccurlyeq [s,\eta]$. Thanks to the $\preccurlyeq$-lower semicontinuity of $f$,
we have $\underline{f}(y_0)\leq s$. Consequently, $\underline{f}$ is lower semicontinuous on $X$.
\end{proof}

\begin{remark}\label{re2.7} Proposition \ref{p2.3} shows that $f+g$ is $\preccurlyeq$-lower semicontinuous on $X$ providing $f,g: X\to \mathcal{I}$ are two $\preccurlyeq$-lower semicontinuous interval-valued functions.
\end{remark}

We also recall the concept of continuity for interval-valued functions as follows.
\begin{definition}(\cite{Wu})\label{de2.6}
An interval-valued function $f:X\to \mathcal{I}$ is said to be Hausdorff continuous at $x_0\in X$ if $\lim_{x\rightarrow x_0}f(x)=f(x_0)$. An interval-valued function $f:X\to \mathcal{I}$ is said to be Hausdorff continuous on $X$ if $f$ is Hausdorff continuous at every point in $X$.
\end{definition}
\begin{remark}\label{re2.8}
Proposition 3.3 in \cite{Wu} shows that $f$ is Hausdorff continuous if and only if $\underline{f}$ and $\overline{f}$ are continuous.
Moreover, it follows from Proposition \ref{p2.3} that the Hausdorff continuity implies the $\preccurlyeq$-lower semicontinuity.
\end{remark}

\begin{remark}\label{re2.9}
From Remark \ref{re2.2}, we can see that Definition \ref{de2.6} corresponds to the $gH$-continuity in the sense of Definition 2.7 in \cite{GCMD}.
\end{remark}

Next we consider the following interval-valued optimization problem:
$$
\textrm{(IOP)}\quad \textrm{Minimize}\; f(x)\quad \textrm{subject to} \;\; x\in X_0,
$$
where $X_0$ is a nonempty set and $f: X_0\to \mathcal{I}$ is an interval-valued function with $f(x)=[\underline{f}(x),\overline{f}(x)]$.

We recall the following concept of minimal solution for (IOP).
\begin{definition}(\cite{OCHR})\label{de2.7}
A point $x_0\in X_0$ is said to be a minimal solution of (IOP) if  there does not exist $x\in X_0$ such that
$f(x)\prec f(x_0)$, that is, for any $x\in X_0$, $f(x)\not\prec f(x_0)$.
\end{definition}

The set of all minimal solutions of (IOP) is denoted by $\textrm{Min}(f, X_0)$.

\begin{remark}\label{re2.10}
If at least one of $\underline{f}(x_0)=\inf_{x\in X_0}\underline{f}(x)$ and $\overline{f}(x_0)=\inf_{x\in X_0}\overline{f}(x)$ holds, then  $x_0\in \textrm{Min}(f, X_0)$. In fact, if $x_0\not\in \textrm{Min}(f, X_0)$, then Definition \ref{de2.7} shows that there exists $x'\in X_0$ such that $f(x')\prec f(x_0)$. Thus,  $\underline{f}(x')<\underline{f}(x_0)$ and $\overline{f}(x')<\overline{f}(x_0)$, which is a contradiction.
\end{remark}
The following example shows that the converse of Remark \ref{re2.10} is not true.
\begin{example}\label{ex2.5}
Let $X_0=\mathbb{R}$ and $f:X_0\to \mathcal{I}$ be an interval-valued function defined by
\begin{equation}
f(x)=
\begin{cases}
[1,2] & \text{if $x>0,$}\\
[0,3] & \text{if $x=0,$}\\
[-1,4] &\text{if $x<0.$}\nonumber
\end{cases}
\end{equation}
Clearly, $x=0$ is a minimal solution of (IOP) and
\begin{equation}
\begin{aligned}
\underline{f}(x)=
\begin{cases}
1 & \text{if $x>0,$}\\
0 & \text{if $x=0,$}\\
-1 &\text{if $x<0,$}\nonumber
\end{cases}
\end{aligned}
\quad
\begin{aligned}
\overline{f}(x)=
\begin{cases}
2 & \text{if $x>0,$}\\
3 & \text{if $x=0,$}\\
4 &\text{if $x<0.$}\nonumber
\end{cases}
\end{aligned}
\end{equation}
However, we can see that $0\not\in \textrm{argmin}_{x\in X_0}\underline{f}(x)\cup \textrm{argmin}_{x\in X_0}\overline{f}(x)$.
\end{example}

\begin{remark}\label{re2.11}
\begin{itemize}
\item [(i)] If there exists $x_0\in X_0$ such that $f(x_0)=\textrm{Inf}_{x\in X_0}f(x)$, then Remark \ref{re2.10} shows that $x_0\in \textrm{Min}(f, X_0)$ .
\item[(ii)] Clearly, if $X_0$ is compact and $f$ is $\preccurlyeq$-lower semicontinuous on $X_0$, then $\textrm{Min}(f,X_0)\neq\emptyset$.
\end{itemize}
\end{remark}

We close this section by recalling the following lemma which will be used in next section.
\begin{lemma}\label{le2.1}(\cite{DHM})
Let $\Gamma:X\rightrightarrows X$ be a set-valued mapping. Suppose that the following conditions are satisfied:
\begin{itemize}
\item[(i)] for each $x\in X$, one has $x\in \Gamma(x)$, and $\Gamma(x)$ is closed;
\item[(ii)] for each $y\in \Gamma(x)$, $\Gamma(y)\subseteq \Gamma(x)$;
\item[(iii)] for any sequence $\{x_n\}$ in $X$ such that $x_{n+1}\in \Gamma(x_n)$ for all $n=1,2,\cdots$, one has $d(x_n,x_{n+1})\rightarrow0$.
\end{itemize}
Then there exists $\overline{x}\in X$ such that $\Gamma(\overline{x})=\{\overline{x}\}$.
\end{lemma}

\section{Ekeland's variational principle for interval-valued functions}
\label{sec3}
In this section, we establish some new versions of Ekeland's variational principle for interval-valued functions.

\begin{theorem}\label{th3.1}
Let $(X,d)$ be a complete metric space and $f:X\to \mathcal{I}$ be an interval-valued function with $f(x)=[\underline{f}(x),\overline{f}(x)]$. Suppose that
\begin{itemize}
\item[(i)] $f$ is $\preccurlyeq$-lower semicontinuous;
\item[(ii)] $f$ is $\preccurlyeq$-lower bounded.
\end{itemize}
Then for every $\varepsilon>0$ and every $x_0\in X$ satisfying $\underline{f}(x_0)\leq\inf_{x\in X}\underline{f}(x)+\varepsilon$ and $\overline{f}(x_0)\leq\inf_{x\in X}\overline{f}(x)+\varepsilon$,
there exists $\overline{x}\in X$ such that
\begin{itemize}
\item[(a)] $f(\overline{x})\preccurlyeq f(x_0)$;
\item[(b)] $d(x_0,\overline{x})\leq1$;
\item[(c)] $\forall x\neq \overline{x}$,\quad $f(x)+[\varepsilon d(x,\overline{x}),\varepsilon d(x,\overline{x})]\not\preccurlyeq f(\overline{x})$.
\end{itemize}
\end{theorem}
\begin{proof}
For every $\varepsilon>0$ and every $x_0\in X$, we define
$$S_0:=\{y\in X: f(y)+[\varepsilon d(x_0,y),\varepsilon d(x_0,y)]\preccurlyeq f(x_0) \}.$$
Thanks to the fact that $\underline{f}(x_0)+\varepsilon d(x_0,x_0)\leq \underline{f}(x_0)$ and $\overline{f}(x_0)+\varepsilon d(x_0,x_0)\leq \overline{f}(x_0)$, we must have $x_0\in S_0$. By the lower semicontinuity of $d(x_0,\cdot)$, it follows from  Proposition \ref{p2.3} and Remark \ref{re2.7} that $f(\cdot)+[\varepsilon d(x_0,\cdot),\varepsilon d(x_0,\cdot)]$ is $\preccurlyeq$-lower semicontinuous and so $S_0$ is closed. This shows that $(S_0,d)$ is a complete metric space.

If $S_0=\{x_0\}$, then we take $\overline{x}:=x_0$, which says that the proof is finished.

Otherwise, we suppose that $S_0\neq\{x_0\}$. Then we can define a set-valued mapping $S: S_0\rightrightarrows S_0$ by
$$
S(x)=\{y\in S_0: f(y)+[\varepsilon d(x,y),\varepsilon d(x,y)]\preccurlyeq f(x)\},\quad\forall x\in S_0.
$$
Clearly, for any $x\in S_0$, we know that $x\in S(x)$ and $S(x)$ is closed due to the closedness of $S_0$.
Now we claim that, for any $x\in S_0$ and $y\in S(x)$, $S(y)\subseteq S(x)$. In fact, for any $y\in S(x)$ and $z\in S(y)$, one has
$f(y)+[\varepsilon d(x,y),\varepsilon d(x,y)]\preccurlyeq f(x)$ and $f(z)+[\varepsilon d(y,z),\varepsilon d(y,z)]\preccurlyeq f(y)$. In view of Remark \ref{re2.3} and Proposition \ref{p2.1}, we can see that $f(z)+[\varepsilon d(x,y)+\varepsilon d(y,z),\varepsilon d(x,y)+\varepsilon d(y,z)] \preccurlyeq f(x)$ and so $$\underline{f}(z)+\varepsilon d(x,z)\leq\underline{f}(z)+\varepsilon d(x,y)+\varepsilon d(y,z)\leq \underline{f}(x), \quad \overline{f}(z)+\varepsilon d(x,z)\leq \overline{f}(z)+\varepsilon d(x,y)+\varepsilon d(y,z) \leq \overline{f}(x).$$
Hence, of course, $f(z)+[\varepsilon d(x,z),\varepsilon d(x,z)]\preccurlyeq f(x)$, which implies that $z\in S(x)$. In other words, $S(y)\subseteq S(x)$ for all $x\in S_0$ and $y\in S(x)$. Next we can take a sequence $\{x_n\}\subseteq S_0$ such that $x_{n+1}\in S(x_n)$ for $n=1,2,\cdots$. Then we have $f(x_{n+1})+[\varepsilon d(x_n,x_{n+1}),\varepsilon d(x_n,x_{n+1})]\preccurlyeq f(x_n)$ and so
$$
d(x_n,x_{n+1})\leq \frac{1}{\varepsilon}(\underline{f}(x_n)-\underline{f}(x_{n+1})),\quad
d(x_n,x_{n+1})\leq \frac{1}{\varepsilon}(\overline{f}(x_n)-\overline{f}(x_{n+1})).
$$
Thus, $\underline{f}(x_{n+1})\leq\underline{f}(x_n)$ and $\overline{f}(x_{n+1})\leq\overline{f}(x_n)$. Moreover, by Remark \ref{re2.6}, we know that $\{\underline{f}(x_n)\}$ and $\{\overline{f}(x_n)\}$ are decreasing and lower bounded sequences and so  $d(x_n,x_{n+1})\rightarrow0$.
The hypotheses of Lemma \ref{le2.1} are then satisfied, so there exists $\overline{x}\in S_0$ satisfying $S(\overline{x})=\{\overline{x}\}$. Thus, $f(\overline{x})+[\varepsilon d(x_0,\overline{x}),\varepsilon d(x_0,\overline{x})]\preccurlyeq f(x_0)$ and $f(x)+[\varepsilon d(x,\overline{x}),\varepsilon d(x,\overline{x})]\not\preccurlyeq f(\overline{x})$ for all $x\in S_0\backslash\{\overline{x}\}$. Since $f(\overline{x})+[\varepsilon d(x_0,\overline{x}),\varepsilon d(x_0,\overline{x})]\preccurlyeq f(x_0)$, we get
$\underline{f}(\overline{x})\leq\underline{f}(\overline{x})+\varepsilon d(x_0,\overline{x})\leq\underline{ f}(x_0)$ and $\overline{f}(\overline{x})\leq\overline{f}(\overline{x})+\varepsilon d(x_0,\overline{x})\leq \overline{f}(x_0)$. In other words,
$f(\overline{x})\preccurlyeq f(x_0)$, which shows the part (a). Now, going back to $f(\overline{x})+[\varepsilon d(x_0,\overline{x}),\varepsilon d(x_0,\overline{x})]\preccurlyeq f(x_0)$, by the hypotheses, one has $\varepsilon d(x_0,\overline{x})\leq \underline{f}(x_0)-\underline{f}(\overline{x})\leq\varepsilon$ and $\varepsilon d(x_0,\overline{x})\leq \overline{f}(x_0)-\overline{f}(\overline{x})\leq\varepsilon$. Then  $d(x_0,\overline{x})\leq1$, which shows the part (b).

Note that $f(x)+[\varepsilon d(x,\overline{x}),\varepsilon d(x,\overline{x})]\not\preccurlyeq f(\overline{x})$ for all $x\in S_0\backslash\{\overline{x}\}$. In order to prove the part (c), it remains to show that $f(x)+[\varepsilon d(x,\overline{x}),\varepsilon d(x,\overline{x})]\not\preccurlyeq f(\overline{x})$ for all $x\in X\backslash S_0$. If it is not true, then there would be some $v\in X\backslash S_0$ such that $f(v)+[\varepsilon d(v,\overline{x}),\varepsilon d(v,\overline{x})]\preccurlyeq f(\overline{x})$. Thanks to $f(\overline{x})+[\varepsilon d(x_0,\overline{x}),\varepsilon d(x_0,\overline{x})]\preccurlyeq f(x_0)$, similar to the previous proofs, we can deduce that $f(v)+[\varepsilon d(x_0,v),\varepsilon d(x_0,v)]\leq f(x_0)$, which contradicts the fact that $v\not\in S_0$. Thus,  one has $f(x)+[\varepsilon d(x,\overline{x}),\varepsilon d(x,\overline{x})]\not\preccurlyeq f(\overline{x})$ for $x\neq x_0$, which shows the part (c).
\end{proof}
\begin{remark}\label{re3.1}
Clearly, if $f$ is a real-valued function, that is, $\underline{f}=\overline{f}$, then Theorem \ref{th3.1} is reduced to Theorem 1 in \cite{Ekeland2}.
\end{remark}

\begin{remark}\label{re3.2}
The part (c) of Theorem \ref{th3.1} implies that, for any $x\in X$, $f(x)+[\varepsilon d(x,\overline{x}),\varepsilon d(x,\overline{x})]\not\prec f(\overline{x})$. Indeed, in view of  Remark \ref{re2.4}, we only need to show that $f(x)+[\varepsilon d(x,\overline{x}),\varepsilon d(x,\overline{x})]\not\prec f(\overline{x})$ holds if $x=\overline{x}$. Suppose to the contrary that $f(\overline{x})+[\varepsilon d(\overline{x},\overline{x}),\varepsilon d(\overline{x},\overline{x})]\prec f(\overline{x})$. Then
$\underline{f}(\overline{x})< \underline{f}(\overline{x})$ and $\overline{f}(\overline{x})<\overline{f}(\overline{x})$, which is a contradiction.
\end{remark}

\begin{remark}\label{re3.3}
Using Remark \ref{re3.2},  Theorem \ref{th3.1} indicates that $\overline{x}$ is a minimal solution of (IOP) with the objective function $f(\cdot)+[\varepsilon d(\cdot,\overline{x}),\varepsilon d(\cdot,\overline{x})]$ defined on $X$.
\end{remark}

Now we give the following example to illustrate Theorem \ref{th3.1}.
\begin{example}
Let $X=\mathbb{R}$ and $f:X\to \mathcal{I}$ be an interval-valued function defined by
\begin{equation}
f(x)=
\begin{cases}
[e^x,e^x+1] & \text{if $x<0,$}\\
[\frac{1}{2},\frac{3}{2}] & \text{if $x=0,$}\\
[e^{-x},e^{-x}+1] &\text{if $x>0.$}\nonumber
\end{cases}
\end{equation}
Then we can see that
\begin{equation}
\begin{aligned}
\underline{f}(x)=
\begin{cases}
e^x & \text{if $x<0,$}\\
\frac{1}{2} & \text{if $x=0,$}\\
e^{-x} &\text{if $x>0,$}\nonumber
\end{cases}
\end{aligned}
\quad
\begin{aligned}
\overline{f}(x)=
\begin{cases}
e^x+1 & \text{if $x<0,$}\\
\frac{3}{2} & \text{if $x=0,$}\\
e^{-x}+1 &\text{if $x>0.$}\nonumber
\end{cases}
\end{aligned}
\end{equation}
Also, we can see that $f$ is $\preccurlyeq$-lower semicontinuous and $\preccurlyeq$-lower bounded. Thus, all the conditions of Theorem \ref{th3.1} are satisfied. Moreover, we know that $\inf_{x\in \mathbb{R}}\underline{f}(x)=0$ and $\inf_{x\in \mathbb{R}}\overline{f}(x)=1$. Take $0<\varepsilon<\frac{1}{2}$. Let $x_0\in\mathbb{R}$ be such that
$$f(x_0)\preccurlyeq\left[\inf_{x\in \mathbb{R}}\underline{f}(x)+\varepsilon,\inf_{x\in \mathbb{R}}\overline{f}(x)+\varepsilon\right].$$
Then $f(x_0)\preccurlyeq[\varepsilon,\varepsilon+1]$. Clearly, $x_0\neq0$. Assume that $x_0<0$. Then  $x_0\leq\ln \varepsilon$. Setting $\overline{x}=x_0-1$, we have $f(\overline{x})=[e^{\overline{x}},e^{\overline{x}}+1]\preccurlyeq f(x_0)$ and $|x_0-\overline{x}|=1$. Now we claim that, for any $x\neq\overline{x}$, $f(x)+[\varepsilon|x-\overline{x}|,\varepsilon|x-\overline{x}|]\not\preccurlyeq f(\overline{x})$. When $x<\overline{x}$, we have $f(x)+[\varepsilon|x-\overline{x}|,\varepsilon|x-\overline{x}|]=[e^x+\varepsilon (\overline{x}-x),e^x+\varepsilon(\overline{x}-x)+1]$ and so $e^x+\varepsilon(\overline{x}-x)>e^{\overline{x}}$. This shows that $ f(\overline{x})=[e^{\overline{x}},e^{\overline{x}}+1]\prec f(x)+[\varepsilon|x-\overline{x}|,\varepsilon|x-\overline{x}|]$ for $x<\overline{x}$. When $x>\overline{x}$, we have the following four cases.

Case 1. If $\overline{x}<x<0$, then $f(x)+[\varepsilon|x-\overline{x}|,\varepsilon|x-\overline{x}|]=[e^x+\varepsilon (x-\overline{x}),e^x+\varepsilon(x-\overline{x})+1]$. It follows from  $e^x>e^{\overline{x}}$ that $ f(\overline{x})=[e^{\overline{x}},e^{\overline{x}}+1]\prec f(x)+[\varepsilon|x-\overline{x}|,\varepsilon|x-\overline{x}|]$ for $\overline{x}<x<0$.

Case 2. If $\overline{x}<x=0$, then $f(x)+[\varepsilon|x-\overline{x}|,\varepsilon|x-\overline{x}|]=[\frac{1}{2}+\varepsilon (1-x_0),\frac{3}{2}+\varepsilon(1-x_0)]$. Then we can calculate $\frac{1}{2}+\varepsilon (1-x_0)>\varepsilon(2-x_0)>e^{\overline{x}}$. Thus,
$f(\overline{x})=[e^{\overline{x}},e^{\overline{x}}+1]\prec f(x)+[\varepsilon|x-\overline{x}|,\varepsilon|x-\overline{x}|]$ for $x=0$.

Case 3. If $\overline{x}<0<x<1-x_0$, then
$$f(x)+[\varepsilon|x-\overline{x}|,\varepsilon|x-\overline{x}|]=[e^{-x}+\varepsilon (x-\overline{x}),e^{-x}+\varepsilon(x-\overline{x})+1].$$
Thanks to the fact that $e^{-x}>e^{\overline{x}}$, it follows that
$f(\overline{x})=[e^{\overline{x}},e^{\overline{x}}+1]\prec f(x)+[\varepsilon|x-\overline{x}|,\varepsilon|x-\overline{x}|]$ for $0<x<1-x_0$.

Case 4. If $\overline{x}<0<1-x_0\le x$, then $f(x)+[\varepsilon|x-\overline{x}|,\varepsilon|x-\overline{x}|]=[e^{-x}+\varepsilon (x-\overline{x}),e^{-x}+\varepsilon(x-\overline{x})+1]$. It is easy to check that
$$e^{-x}+\varepsilon (x-\overline{x})\geq e^{x_0-1}+2\varepsilon(1-x_0)>e^{x_0-1}=e^{\overline{x}}$$
and so $f(\overline{x})=[e^{\overline{x}},e^{\overline{x}}+1]\prec f(x)+[\varepsilon|x-\overline{x}|,\varepsilon|x-\overline{x}|]$ for $x>1-x_0$.

Consequently, we conclude that, when $x_0<0$, for $x\neq\overline{x}$, $f(x)+[\varepsilon|x-\overline{x}|,\varepsilon|x-\overline{x}|]\not\preccurlyeq f(\overline{x})$. Similarly, we can show that, when $x_0>0$, for $x\neq\overline{x}$, $f(x)+[\varepsilon|x-\overline{x}|,\varepsilon|x-\overline{x}|]\not\preccurlyeq f(\overline{x})$.
\end{example}

The following corollary is immediately obtained from Theorem \ref{th3.1}.
\begin{corollary}\label{co3.1}
Let $(X,d)$ be a complete metric space and $f:X\to \mathcal{I}$ be an interval-valued function with $f(x)=[\underline{f}(x),\overline{f}(x)]$. Suppose that $f$ is $\preccurlyeq$-lower semicontinuous and $\preccurlyeq$-lower bounded.
Then for every $\varepsilon>0$, there exists $x_{\varepsilon}\in X$ such that
\begin{itemize}
\item[(a)] $f(x_{\varepsilon})\preccurlyeq [\inf_{x\in X}\underline f(x)+\varepsilon,\inf_{x\in X}\overline f(x)+\varepsilon]$;
\item[(b)] $\forall x\neq {x_\varepsilon}$, $f(x)+[\varepsilon d(x,{x_\varepsilon}),\varepsilon d(x,{x_\varepsilon})]\not\preccurlyeq f(x_\varepsilon)$.
\end{itemize}
\end{corollary}

Now we turn to establish new versions of Ekeland's variational principle involving the generalized Hukuhara G\^{a}teaux differentiability of interval-valued functions. To this end, we need the following definitions.

\begin{definition}\label{de3.1}(\cite{GCMD})
Let $X$ be a Banach space.  An interval-valued function $f: X\to \mathcal{I}$ is said to be linear if
\begin{itemize}
\item[(i)]  for any $x\in X$ and $\lambda\in\mathbb{R}$, $f(\lambda x)=\lambda f(x)$;
\item[(ii)] for any $x,y\in X$, either $f(x+y)=f(x)+f(y)$, or $f(x+y)$ and $f(x)+f(y)$ are incomparable in the sense of (i) in Definition \ref{de2.2}.
\end{itemize}
\end{definition}

Based on the $gH$-difference, the following concept of G\^{a}teaux differentiability of interval-valued functions was introduced by \cite{GCMD}.
\begin{definition} (\cite{GCMD}) \label{de3.2}
Let $X$ be a Banach space and $f: X\to \mathcal{I}$ be an interval-valued function with $f(x)=[\underline{f}(x),\overline{f}(x)]$. We say that $f$ is $gH$-G\^{a}teaux differentiable at $x_0\in X$ if, for every $h\in X$ such that
$$
f'(x_0)(h):=\lim_{t\rightarrow0_+}\frac{f(x_0+th)\ominus_{gH} f(x_0)}{t}
$$
exists, and $f'(x_0)$ is a Hausdorff continuous and linear interval-valued function from $X$ to $\mathcal{I}$.
The interval-valued function $f$ is said to be $gH$-G\^{a}teaux differentiable on $X$ if $f$ is $gH$-G\^{a}teaux differentiable at every point in $X$.
\end{definition}

\begin{theorem}\label{th3.2}
Let $X$ be a Banach space and $f:X\to \mathcal{I}$ be an interval-valued function with $f(x)=[\underline{f}(x),\overline{f}(x)]$. Suppose that
\begin{itemize}
\item[(i)] $f$ is $gH$-G\^{a}teaux differentiable;
\item[(ii)] $f$ is $\preccurlyeq$-lower semicontinuous and $\preccurlyeq$-lower bounded.
\end{itemize}
Then for any $\varepsilon>0$, there exists $x_{\varepsilon}\in X$ such that
\begin{itemize}
\item[(a)] $f(x_{\varepsilon})\preccurlyeq [\inf_{x\in X}\underline f(x)+\varepsilon,\inf_{x\in X}\overline f(x)+\varepsilon]$;
\item[(b)]$\forall h\in X$, $\|h\|_X=1$, $f'(x_{\varepsilon})(h)\not\prec [-\varepsilon,-\varepsilon]$, where $\|\cdot\|_{X}$ is the norm on $X$.
\end{itemize}
\end{theorem}
\begin{proof}
According to Corollary \ref{co3.1} and Remark \ref{re3.2}, we know that there exists $x_{\varepsilon}\in X$ such that
\begin{equation}\label{eq3.1}
f(x_{\varepsilon})\preccurlyeq \left[\inf_{x\in X}\underline f(x)+\varepsilon,\inf_{x\in X}\overline f(x)+\varepsilon\right]
\end{equation}
 and
\begin{equation}\label{eq3.2}
\forall x\in X,~f(x)+[\varepsilon\|x-x_{\varepsilon}\|_X,\varepsilon\|x-x_{\varepsilon}\|_X]\not\prec f(x_{\varepsilon}).
\end{equation}
Thus, the part (a) is obtained by (\ref{eq3.1}).  For any $h\in X$ with $\|h\|_X=1$, set $x=x_{\varepsilon}+th$ with $t>0$. Then by Remark \ref{re2.5} and (\ref{eq3.2}), we have
$$
\frac {f(x_{\varepsilon}+th)\ominus_{gH}f(x_{\varepsilon})}{t}\not\prec [-\varepsilon,-\varepsilon].
$$
Moreover, applying Proposition \ref{p2.2} and $gH$-G\^{a}teaux differentiability of $f$, one has
$$
f'(x_{\varepsilon})(h)=\lim_{t\rightarrow0_+}\frac {f(x_{\varepsilon}+th)\ominus_{gH}f(x_{\varepsilon})}{t}\not\prec [-\varepsilon,-\varepsilon],
$$
which shows the part (b).
\end{proof}
\begin{remark}\label{re3.4}
If $f$ is a real-valued function, then we can deduce that Theorem \ref{th3.2} corresponds to Theorem 8 in \cite{Ekeland2}.
\end{remark}

\begin{theorem}\label{th3.3}
Let $X$ be a Banach space and $f:X\to \mathcal{I}$ be an interval-valued function with $f(x)=[\underline{f}(x),\overline{f}(x)]$. Suppose that
\begin{itemize}
\item[(i)] $f$ is $gH$-G\^{a}teaux differentiable;
\item[(ii)] $f$ is $\preccurlyeq$-lower semicontinuous and $\preccurlyeq$-lower bounded.
\end{itemize}
Then there exists a sequence $\{x_n\}$ in $X$ such that
\begin{itemize}
\item[(a)] $f(x_n)\rightarrow [\inf_{x\in X}\underline f(x),\inf_{x\in X}\overline f(x)]$;
\item[(b)]$\forall h\in X$, $\|h\|_X=1$, $f'(x_n)(h)\rightarrow[0,0]$  or $0\in f'(x_n)(h)$ for $n=1,2,\cdots$.
\end{itemize}
\end{theorem}
\begin{proof}
Taking $\varepsilon_n>0$ with $\varepsilon_n\rightarrow0$, it follows from Theorem \ref{th3.2} that there exists a sequence $\{x_n\}$ in $X$ such that
\begin{equation}\label{eq3.3}
f(x_n)\preccurlyeq \left[\inf_{x\in X}\underline f(x)+\varepsilon_n,\inf_{x\in X}\overline f(x)+\varepsilon_n\right]
\end{equation}
and
\begin{equation}\label{eq3.4}
\forall h\in X,\quad \|h\|_X=1,\quad f'(x_n)(h)\not\prec [-\varepsilon_n,-\varepsilon_n ].
\end{equation}
By (\ref{eq3.3}), we have
$$
\inf_{x\in X}\underline f(x)\leq\underline f(x_n)\leq \inf_{x\in X}\underline f(x)+\varepsilon_n, \quad \inf_{x\in X}\overline{f}(x)\leq \overline{f}(x_n)\leq\inf_{x\in X}\overline f(x)+\varepsilon_n.
$$
It follows that $\underline f(x_n)\rightarrow\inf_{x\in X}\underline f(x)$ and $\overline{f}(x_n)\rightarrow\inf_{x\in X}\overline{f}(x)$. Thus, $f(x_n)\rightarrow [\inf_{x\in X}\underline f(x),\inf_{x\in X}\overline f(x)]$. We assume that $f'(x_n)(h):=[\underline{g}(x_n)(h),\overline{g}(x_n)(h)]$, where $\underline{g}(\cdot)(h)$ and $\overline{g}(\cdot)(h)$ depend on $f'(\cdot)(h)$. Then by (\ref{eq3.4}) and the definition of $\not\prec$ and the fact that $\underline{g}(x_n)(h)\leq\overline{g}(x_n)(h)$ , we have the following two cases:

Case 1. $\underline{g}(x_n)(h)< -\varepsilon_n$, $\overline{g}(x_n)(h)\geq-\varepsilon_n$;

Case 2. $\underline{g}(x_n)(h)\geq -\varepsilon_n$, $\overline{g}(x_n)(h)\geq-\varepsilon_n$.

From the linearity of $f'(x_n)$, one has $f'(x_n)(-h)=-f'(x_n)(h)$. Thus, $\underline{g}(x_n)(-h)=-\overline{g}(x_n)(h)$ and
$\overline{g}(x_n)(-h)=-\underline{g}(x_n)(h)$. For Case 1, we can see that
$\underline{g}(x_n)(h)< -\varepsilon_n$ and $\overline{g}(x_n)(h)>\varepsilon_n$ and so $0\in f'(x_n)(h)$ for $n=1,2,\cdots$. For Case 2, it follows that $-\varepsilon_n\leq\underline{g}(x_n)(h)\leq\overline{g}(x_n)(h)\leq\varepsilon_n$, which says that $\underline{g}(x_n)(h)\rightarrow0$ and $\overline{g}(x_n)(h)\rightarrow0$. Namely, $f'(x_n)(h)\rightarrow[0,0]$. Consequently, we conclude that the part (b) holds.
\end{proof}

\begin{remark}\label{re3.5}
If $f$ is a real-valued function, then the part (b) of Theorem \ref{th3.3} is reduced to $f'(x_n)(h)\rightarrow0$.
\end{remark}

Now we give the following example to illustrate Theorem \ref{th3.3}.
\begin{example}
Let $X=\mathbb{R}$ and $f:X\to \mathcal{I}$ be an interval-valued function defined by
$$f(x)=\left[\frac{1}{x^2+1},\frac{1}{x^2+1}+1\right],\quad \forall x\in\mathbb{R}.$$
Then we see that $\underline{f}(x)=\frac{1}{x^2+1}$ and $\overline{f}(x)=\frac{1}{x^2+1}+1$ and so $f$ is Hausdorff continuous. Moreover, it follows that $\inf_{x\in \mathbb{R}}\underline{f}(x)=0$ and $\inf_{x\in \mathbb{R}}\overline{f}(x)=1$. Also, $[0,1]\preccurlyeq f(x)$ for all $x\in\mathbb{R}$.
For any $x\in \mathbb{R}$, $h\in \mathbb{R}$ and $t>0$,
\begin{align}
f'(x)(h):&=\lim_{t\rightarrow0_+}\frac{f(x+th)\ominus_{gH}f(x)}{t}\nonumber\\
&=\lim_{t\rightarrow0_+}\frac{1}{t}\left(\left[\frac{1}{(x+th)^2+1},\frac{1}{(x+th)^2+1}+1\right]\ominus_{gH}\left[\frac{1}{x^2+1},\frac{1}{x^2+1}+1\right]\right)\nonumber\\
&=\lim_{t\rightarrow0_+}\frac{1}{t}\left[\frac{-2thx-t^2h^2}{((x+th)^2+1)(x^2+1)},\frac{-2thx-t^2h^2}{((x+th)^2+1)(x^2+1)}\right]\nonumber\\
&=\left[\frac{-2hx}{(x^2+1)^2},\frac{-2hx}{(x^2+1)^2}\right].\nonumber
\end{align}
It is easy to see that $f'(x)(\cdot)$ is Hausdorff continuous and linear on $\mathbb{R}$ and so $f$ is $gH$-G\^{a}teaux differentiable on $\mathbb{R}$. Thus, we know that all the conditions of Theorem \ref{th3.3} are satisfied.
Take $x_n=n$ for $n=1,2,\cdots$. When $n\rightarrow+\infty$, one has $f(x_n)\rightarrow[0,1]=[\inf_{x\in\mathbb{R}}\underline{f}(x),\inf_{x\in\mathbb{R}}\overline{f}(x)]$ and $f'(x_n)(h)\rightarrow[0,0]$ for all $h\in\mathbb{R}$ with $|h|=1$.
\end{example}

The following theorem establishes a new version of Ekeland's variational principle for interval-valued bifunctions.
\begin{theorem}\label{th3.4}
Let $(X,d)$ be a complete metric space and $f:X\times X\to \mathcal{I}$ be an interval-valued bifunction with $f(x,y)=[\underline{f}(x,y),\overline{f}(x,y)]$. Assume that
\begin{itemize}
\item[(i)] for each $x\in X$, $f(x,\cdot)$ is $\preccurlyeq$-lower semicontinuous and $\preccurlyeq$-lower bounded;
\item[(ii)] $f$ satisfies the triangle inequality property, i.e., for any $x,y,z\in X$, $f(x,z)\preccurlyeq f(x,y)+f(y,z)$.
\end{itemize}
Then for every $\varepsilon>0$ and every $x_0\in X$, there exists $\overline{x}\in X$ such that
\begin{itemize}
\item[(a)] $f(x_0,\overline{x})\preccurlyeq f(x_0,x_0)$;
\item[(b)] $\forall x\neq \overline{x}$, $f(\overline{x},x)+[\varepsilon d(x,\overline{x}),\varepsilon d(x,\overline{x})]\not\preccurlyeq [0,0]$.
\end{itemize}
\end{theorem}
\begin{proof}
For fixed $x_0\in X$, let $F_{x_0}(x)=f(x_0,x)$ for all $x\in X$. Then applying Corollary \ref{co3.1} to $F_{x_0}$,  there exists $\overline{x}\in X$ such that
$$
F_{x_0}(\overline{x})\preccurlyeq F_{x_0}(x_0),
$$
$$
\forall x\neq \overline{x},\quad F_{x_0}(x)+[\varepsilon d(x,\overline{x}),\varepsilon d(x,\overline{x})]\not\preccurlyeq F_{x_0}(\overline{x}),
$$
that is,
\begin{equation}\label{eq3.5}
f(x_0,\overline{x})\preccurlyeq f(x_0,x_0),
\end{equation}
\begin{equation}\label{eq3.6}
\forall x\neq \overline{x},\quad f(x_0,x)+[\varepsilon d(x,\overline{x}),\varepsilon d(x,\overline{x})]\not\preccurlyeq f(x_0,\overline{x}).
\end{equation}
It follows from (\ref{eq3.5}) that the part (a) holds. Now we claim that the part (b) is true. Indeed, if not, then there exists $x'\in X\backslash\{\overline{x}\}$ such that
$f(\overline{x},x')+[\varepsilon d(x',\overline{x}),\varepsilon d(x',\overline{x})]\preccurlyeq[0,0]$.
Thanks to the fact that $f(x_0,x')\preccurlyeq f(x_0,\overline{x})+f(\overline{x},x')$, Proposition \ref{p2.1} leads to
$$
f(x_0,x')+[\varepsilon d(x',\overline{x}),\varepsilon d(x',\overline{x})]\preccurlyeq f(x_0,\overline{x})+f(\overline{x},x')+[\varepsilon d(x',\overline{x}),\varepsilon d(x',\overline{x})]\preccurlyeq f(x_0,\overline{x}).
$$
By Remark \ref{re2.3}, we can see that $f(x_0,x')+[\varepsilon d(x',\overline{x}),\varepsilon d(x',\overline{x})]\preccurlyeq f(x_0,\overline{x})$, which contradicts (\ref{eq3.6}).
\end{proof}
\begin{remark}\label{re3.6}
We would like to mention that the triangle inequality property of $f$ implies that $[0,0]\preccurlyeq f(x,x)$ for all $x\in X$. In fact, letting $x=y=z$, we have
$f(x,x)\preccurlyeq f(x,x)+f(x,x)$. This shows that $0\leq \underline{f}(x,x)$ and $0\leq \overline{f}(x,x)$, and so $[0,0]\preccurlyeq f(x,x)$.
\end{remark}

\begin{remark}\label{re3.7}
The part (b) of Theorem \ref{th3.4} implies that $f(\overline{x},x)+[\varepsilon d(x,\overline{x}),\varepsilon d(x,\overline{x})]\not\prec [0,0]$ for all $x\in X$. In fact, by Remark \ref{re2.4}, we only need to prove that $f(\overline{x},x)\not\prec [0,0]$ holds when $x=\overline{x}$. If not, then $f(\overline{x},\overline{x})\prec [0,0]$, which contradicts the fact that $[0,0]\preccurlyeq f(\overline{x},\overline{x})$ due to Remark \ref{re3.6}.
\end{remark}

Now we give the following example to illustrate Theorem \ref{th3.4}.
\begin{example}
Let $X=\mathbb{R}$ and $f: X\times X\to \mathcal{I}$ be an interval-valued bifunction defined by
$$f(x,y)=[|x-y|,|x-y|+1],\;\forall x,y\in\mathbb{R}.$$
Then it is easy to see that all the conditions of Theorem \ref{th3.4} are satisfied. Moreover, we can check that the parts (a) and (b) of Theorem \ref{th3.4} are true.
\end{example}
\section{Applications}
\label{sec4}
In this section, we apply the obtained results in the previous section to fixed point theorems, to interval-valued optimization problems, to the interval-valued mountain pass theorem, to noncooperative interval-valued games, and to interval-valued optimal control problems governed by interval-valued differential equations.
\subsection{Fixed point theorems}
The following result gives Carist's fixed point theorem for set-valued mappings.
\begin{theorem}\label{th4.1}
Let $(X,d)$ be a complete metric space, $f:X\to \mathcal{I}$ be an interval-valued function with $f(x)=[\underline{f}(x),\overline{f}(x)]$, and $T: X\rightrightarrows X$ be a set-valued mapping. Suppose that
\begin{itemize}
\item[(i)] $f$ is $\preccurlyeq$-lower semicontinuous and $\preccurlyeq$-lower bounded;
\item[(ii)] for any $x\in X$ and $y\in T(x)$, $f(y)+[d(x,y),d(x,y)]\preccurlyeq f(x)$.
\end{itemize}
Then $T$ has a fixed point, i.e., there exists $\overline{x}\in X$ such that $\overline{x}\in T(\overline{x})$.
\end{theorem}
\begin{proof}
Applying Corollary \ref{co3.1}, letting $\varepsilon=1$, there exists $\overline{x}\in X$ such that
\begin{equation}\label{eq4.5}
\forall x\neq \overline{x},\quad f(x)+[d(x,\overline{x}), d(x,\overline{x})]\not\preccurlyeq f(\overline{x}).
\end{equation}
If $\overline{x}$ is a fixed point of $T$, then the proof is finished. Suppose to the contrary that $\overline{x}$ is not a fixed point of $T$, that is, $\overline{x}\not\in T(\overline{x})$. Then for any $y\in T(\overline{x})$ with $y\neq \overline{x}$, it follows from condition (ii) that
$f(y)+[d(\overline{x},y),d(\overline{x},y)]\preccurlyeq f(\overline{x})$, which contradicts (\ref{eq4.5}).
\end{proof}

By Theorem \ref{th4.1}, it is easy to obtain the following new fixed point theorem for an interval-valued function.
\begin{theorem}\label{th4.2}
Let $f:\mathbb{R}\to \mathcal{I}$ be an interval-valued function with $f(x)=[\underline{f}(x),\overline{f}(x)]$. Suppose that
\begin{itemize}
\item[(i)] $f$ is $\preccurlyeq$-lower semicontinuous and $\preccurlyeq$-lower bounded;
\item[(ii)] for any $x\in\mathbb{R}$ and $y\in f(x)$, $f(y)+[|x-y|,|x-y|]\preccurlyeq f(x)$.
\end{itemize}
Then $f$ has a fixed point.
\end{theorem}

\subsection{Interval-valued optimization problems}
In this subsection, we first give Takahashi's minimization theorem for an interval-valued function. Then we introduce the Palais-Smale condition for interval-valued functions to obtain the existence of minimal solutions for interval-valued optimization problems.

\begin{theorem}\label{th4.3}
Let $(X,d)$ be a complete metric space and $f:X\to \mathcal{I}$ be an interval-valued function with $f(x)=[\underline{f}(x),\overline{f}(x)]$. Suppose that
\begin{itemize}
\item[(i)] $f$ is $\preccurlyeq$-lower semicontinuous and $\preccurlyeq$-lower bounded;
\item[(ii)] for any $x\not\in\textrm{Min}(f,X)$, there is $y\neq x$ such that $f(y)+[d(x,y),d(x,y)]\preccurlyeq f(x)$.
\end{itemize}
Then $\textrm{Min}(f,X)\neq\emptyset$.
\end{theorem}
\begin{proof}
According to Corollary \ref{co3.1}, letting $\varepsilon=1$, there exists a point $\overline{x}\in X$ such that
\begin{equation}\label{eq4.8}
\forall x\neq \overline{x},\quad f(x)+[d(x,\overline{x}), d(x,\overline{x})]\not\preccurlyeq f(\overline{x}).
\end{equation}
We claim that $\overline{x}\in\textrm{Min}(f,X)$. Indeed, if $\overline{x}\not\in\textrm{Min}(f,X)$, then by condition (ii), there exists $y\neq\overline{x}$ such that $f(y)+[d(\overline{x},y),d(\overline{x},y)]\preccurlyeq f(\overline{x})$, which contradicts (\ref{eq4.8}). Thus, $\overline{x}\in\textrm{Min}(f,X)$ and so $\textrm{Min}(f,X)\neq\emptyset$.
\end{proof}

Now we introduce the concept of critical point for an interval-valued function.
\begin{definition}\label{de4.1}
Let $X_0$ be a nonempty open subset of a Banach space $X$ and $f: X_0\to \mathcal{I}$ be $gH$-G\^{a}teaux differentiable.  A point $x_0\in X_0$ is said to be a critical point of $f$ if, for any $h\in X$, $0\in f'(x_0)(h)$.
\end{definition}
\begin{remark}\label{re4.1}
If $f$ is a real-valued function, then Definition \ref{de4.1} coincides with the classical concept of critical points (see \cite{Chang}).
\end{remark}

The following example illustrates Definition \ref{de4.1}.
\begin{example}\label{ex4.1}
Let $X=\mathbb{R}$, $X_0=(0,1)$ and $f:X_0\to \mathcal{I}$ be an interval-valued function defined by
$$
f(x)=[-x,x],\quad x\in (0,1).
$$
For $h\in \mathbb{R}$, $t>0$, $x\in (0,1)$ with $x+th\in(0,1)$,  we can see that
\begin{equation}
f'(x)(h)=\lim_{t\rightarrow0_+}\frac{f(x+th)\ominus_{gH}f(x)}{t}=
\begin{cases}
[-h,h] & \text{if $h\geq0,$}\\
[h,-h] &\text{if $h<0.$}\nonumber
\end{cases}
\end{equation}
It follows that $f$ is $gH$-G\^{a}teaux differentiable on $(0,1)$. Moreover, for any $h\in \mathbb{R}$, one has $0\in f'(x)(h)$ for all $x\in (0,1)$. This shows that every point in $(0,1)$ is a critical point of $f$.
\end{example}

The following result gives the first-order necessary optimality condition for minimal solutions of (IOP).
\begin{proposition}\label{p4.1}
Let $X_0$ be a nonempty open subset of a Banach space $X$ and $f: X_0\to \mathcal{I}$ be $gH$-G\^{a}teaux differentiable. If $x_0\in\textrm{Min}(f,X_0)$, then $x_0$ is a critical point of $f$.
\end{proposition}
\begin{proof}
Because $x_0\in\textrm{Min}(f,X_0)$, we have $f(x)\not\prec f(x_0)$ for all $x\in X_0$.
In view of Remark \ref{re2.4}, one has $f(x)\ominus_{gH} f(x_0)\not\prec[0,0]$. Let $h\in X$ and $t>0$ be such that $x_0+th\in X_0$. Then it follows from Proposition \ref{p2.2} and the $gH$-G\^{a}teaux differentiability of $f$ that
$$
f'(x_0)(h)=\lim_{t\rightarrow0_+}\frac{f(x_0+th)\ominus_{gH} f(x_0)}{t}\not\prec[0,0].
$$
Assume that $f'(x_0)(h):=[\underline{g}(x_0)(h),\overline{g}(x_0)(h)]$, where $\underline{g}(x_0)(h)$ and $\overline{g}(x_0)(h)$ depend on $f'(x_0)(h)$. Then by the definition of $\not\prec$ and the fact that $\underline{g}(x_0)(h)\leq\overline{g}(x_0)(h)$, we have the following two cases:

Case 1. $\underline{g}(x_0)(h)<0$, $\overline{g}(x_0)(h)\geq0$;

Case 2. $\underline{g}(x_0)(h)\geq0$, $\overline{g}(x_0)(h)\geq0$.

For Case 1, one has $0\in f'(x_0)(h)$. Now we consider Case 2. It follows from the linearity of $f'(x_0)$ that $f'(x_0)(-h)=-f'(x_0)(h)$. Thus, $\underline{g}(x_0)(-h)=-\overline{g}(x_0)(h)$ and $\overline{g}(x_0)(-h)=-\underline{g}(x_0)(h)$. This shows that
$\underline{g}(x_0)(h)=\overline{g}(x_0)(h)=0$. Combining Cases 1 and 2,  we show that $x_0$ is a critical point of $f$.
\end{proof}

In order to introduce the Palais-Smale condition for interval-valued functions,  we use $HC^1(X,\mathcal{I})$ to denote the set of all interval-valued functions $f: X\to \mathcal{I}$  such that $f$ is Hausdorff continuous and $gH$-G\^{a}teaux differentiable,  and for any $h\in X$, $f'(\cdot)(h)$ is Hausdorff continuous.

\begin{definition}\label{de4.2}
Let $f\in HC^1(X,\mathcal{I})$. We say that $f$ satisfies the Palais-Smale condition if any sequence $\{x_n\}$ in $X$ such that
\begin{itemize}
\item[(i)] $f(x_n)$ is $\preccurlyeq$-bounded;
\item[(ii)] for each $h\in X$, $\|h\|_X=1$, $f'(x_n)(h)\rightarrow[0,0]$ or $0\in f'(x_n)(h)$ for $n$ large enough,
\end{itemize}
has a convergent subsequence.
\end{definition}

Clearly, if $f$ is a real-valued function, then Definition \ref{de4.2} coincides with the classical Palais-Smale condition.

The following example illustrates the Palais-Smale condition for interval-valued functions.
\begin{example}\label{ex4.2}
Let $X=\mathbb{R}$ and $f:X\to\mathcal{I}$ be an interval-valued function defined by
$$
f(x)=[-x^2,x^2],\quad\forall x\in X.
$$
Then it is easy to see that $f$ is Hausdorff continuous. For any $x\in \mathbb{R}$, $h\in \mathbb{R}$ with $t>0$, we have
\begin{equation}
f'(x)(h):=\lim_{t\rightarrow0_+}\frac {f(x+th)\ominus_{gH}f(x)}{t}=
\begin{cases}
[-2hx,2hx] & \text{if $hx\geq0,$}\\
[2hx,-2hx] &\text{if $hx<0.$}\nonumber
\end{cases}
\end{equation}
Clearly, $f$ is $gH$-G\^{a}teaux differentiable,  $f'(x)(h)$ is Hausdorff continuous and $0\in f'(x)(h)$ for all $x,h\in \mathbb{R}$. Moreover, if for any sequence $\{x_n\}$ in $\mathbb{R}$ such that $\{f(x_n)\}$ is $\preccurlyeq$-bounded,
 then we can check that $\{x_n\}$ is bounded in $\mathbb{R}$ and so $\{x_n\}$ has a convergent subsequence. Thus, $f$ satisfies the Palais-Smale condition.
\end{example}

Finally, we give the following existence result concerned with minimal solutions for (IOP) under the Palais-Smale condition.
\begin{theorem}\label{th4.4}
Let $f\in HC^1(X,\mathcal{I})$ be $\preccurlyeq$-lower bounded.  If $f$ satisfies the Palais-Smale condition, then
 $\textrm{Min}(f,X)\neq\emptyset$.
\end{theorem}
\begin{proof}
Because $f$ is Hausdorff continuous on $X$, by Remark \ref{re2.8}, we apply Theorem \ref{th3.3} to obtain a sequence $\{x_n\}$ in $X$ such that
\begin{equation*}
f(x_n)\rightarrow \left[\inf_{x\in X}\underline f(x),\inf_{x\in X}\overline f(x)\right]
\end{equation*}
and
$$
\forall h\in X, \quad\|h\|_X=1,\quad f'(x_n)(h)\rightarrow[0,0]\; \textrm{or}\; 0\in f'(x_n)(h)~ \textrm{for}~ n=1,2,\cdots.
$$
Then $\underline{f}(x_n)\rightarrow \inf_{x\in X}\underline f(x)$ and $\overline{f}(x_n)\rightarrow \inf_{x\in X}\overline f(x)$. Thus, $\underline{f}(x_n)$ and $\overline{f}(x_n)$ are bounded and so $f(x_n)$ is $\preccurlyeq$-bounded due to Remark \ref{re2.6}.  Applying the Palais-Smale condition, there must exist a subsequence $\{x_{n_k}\}$ of $\{x_n\}$ converging to $x_0\in X$. By the Hausdorff continuity of $f$, one has $f(x_{n_k})\rightarrow f(x_0)$, which implies that $\underline{f}(x_0)=\inf_{x\in X}\underline f(x)$ and $\overline{f}(x_0)=\inf_{x\in X}\overline f(x)$. It follows from Remark \ref{re2.10} that $x_0\in \textrm{Min}(f,X)$ and so $\textrm{Min}(f,X)\neq\emptyset$.
\end{proof}

\begin{remark}\label{re4.2}
From Proposition \ref{p4.1}, we can see that Theorem \ref{th4.4} implies that there exists a critical point of $f$.
\end{remark}

\subsection{The Mountain Pass Theorem for interval-valued functions}
In this subsection, let $C([0,1],X)$ denote the set of all continuous functions from $[0,1]$ to $X$, where $X$ is a Banach space.
\begin{definition}\label{de4.3}
Let $f\in HC^1(X,\mathcal{I})$ and $C\in\mathcal{I}$. We say that $f$ satisfies the $(\textrm{Palais-Smale})_C$ condition if any sequence $\{x_n\}$ in $X$ such that
\begin{itemize}
\item[(i)] $f(x_n)\rightarrow C$;
\item[(ii)] for each $h\in X$, $\|h\|_X=1$, $f'(x_n)(h)\rightarrow[0,0]$  or $0\in f'(x_n)(h)$ for $n$ large enough,
\end{itemize}
has a convergent subsequence.
\end{definition}
\begin{remark}\label{re4.3}
By Remark \ref{re2.6}, the Palais-Smale condition implies the $(\textrm{Palais-Smale})_C$ condition.
\end{remark}

The following result establishes an interval-valued version of the Mountain Pass Theorem.
\begin{theorem}\label{th4.5}
Let $f\in HC^1(X,\mathcal{I})$ and $\Omega$ be an open subset of $X$. For any given $p_0\in\Omega$ and $p_1\not\in cl\Omega $, where $cl\Omega$ denotes the closure of $\Omega$, set $\Gamma=\{l\in C([0,1],X):l(0)=p_0,\;l(1)=p_1\}$. Assume that
\begin{itemize}
\item[(i)] $f(p_0)\prec\alpha$ and $f(p_1)\prec\alpha$, where $\alpha=[\inf_{x\in\partial\Omega}\underline{f}(x),\inf_{x\in\partial\Omega}\overline{f}(x)]$ and $\partial\Omega$ is the boundary of $\Omega$;
\item[(ii)] $C:=[\inf_{l\in\Gamma}\max_{t\in[0,1]}\underline{f}(l(t)),\inf_{l\in\Gamma}\max_{t\in[0,1]}\overline{f}(l(t))]$ and $f$ satisfies the $(\textrm{Palais-Smale})_C$ condition.
\end{itemize}
Then $C$ is the critical value of $f$, i.e., there exists a point $x_0\in X$ such that $x_0$ is a critical point of $f$ and $f(x_0)=C$.
\end{theorem}
\begin{proof}
Let us define a distance $d$ on $\Gamma$ as follows:
$$
d(l_1,l_2)=\max_{t\in[0,1]}\|l_1(t)-l_2(t)\|_X, \quad \forall l_1,l_2\in \Gamma.
$$
Clearly, $(\Gamma,d)$ is a complete metric space. Now we can define an interval-valued function $\Phi:\Gamma\to\mathcal{I}$ as follows:
$$
\Phi(l)=\left[\max_{t\in[0,1]}\underline{f}(l(t)),\max_{t\in[0,1]}\overline{f}(l(t))\right], \quad \forall l\in\Gamma.
$$
By Remark \ref{re2.8}, we know that $\underline{f}$ and $\overline{f}$ are continuous on $X$. Thus, $\max_{t\in[0,1]}\underline{f}(l(t))$ and $\max_{t\in[0,1]}\overline{f}(l(t))$ exist. It follows that $\max_{t\in[0,1]}\underline{f}(l(t))\leq\max_{t\in[0,1]}\overline{f}(l(t))$ and so $\Phi$ is well-defined. Set $\Phi(l)=[\underline{\Phi}(l),\overline{\Phi}(l)]$, where $\underline{\Phi}(l)=\max_{t\in[0,1]}\underline{f}(l(t))$ and
$\overline{\Phi}(l)=\max_{t\in[0,1]}\overline{f}(l(t))$. By the hypothesis, one has $\underline\Phi(l)\geq\inf_{x\in\partial\Omega}\underline{f}(x)$
and $\overline\Phi(l)\geq\inf_{x\in\partial\Omega}\overline{f}(x)$, and so $\alpha\preccurlyeq \Phi(l)$ for all $l\in\Gamma$.

Now we prove that $\Phi$ is $\preccurlyeq$-lower semicontinuous on $\Gamma$. Invoking Proposition \ref{p2.3}, we only need to show that $\underline{\Phi}$ and $\overline{\Phi}$ are lower semicontinuous on $\Gamma$. For any $r\in \mathbb{R}$, let $\{s_n\}$ be a sequence in $\{l\in\Gamma:\underline{\Phi}(l)\leq r\}$ converging to $s_0$. Then $\underline{\Phi}(s_n)\leq r$ and so $\underline{f}(s_n(t))\leq r$ for all $t\in[0,1]$. Because $\underline{f}$ is continuous, we have $\underline{f}(s_0(t))\leq r$ for all $t\in[0,1]$.  Thus, $\underline{\Phi}(s_0)\leq r$ and so $\underline{\Phi}$ is lower semicontinuous on $\Gamma$. Similarly, we can prove that  $\overline{\Phi}$ is lower semicontinuous on $\Gamma$. Hence $\Phi$ satisfies all the conditions of Corollary \ref{co3.1}, and it follows from Remark \ref{re3.2} that there exists $\{l_n\}$ in $\Gamma$ such that
\begin{equation}\label{eq4.9}
\left[\inf_{l\in\Gamma}\underline{\Phi}(l),\inf_{l\in\Gamma}\overline{\Phi}(l)\right]\preccurlyeq \Phi(l_n)\preccurlyeq \left[\inf_{l\in\Gamma}\underline{\Phi}(l)+\frac{1}{n},\inf_{l\in\Gamma}\overline{\Phi}(l)+\frac{1}{n}\right]
\end{equation}
and
\begin{equation}\label{eq4.10}
\forall l\in\Gamma,\quad\Phi(l)+\left[\frac{1}{n}d(l,l_n),\frac{1}{n}d(l,l_n)\right]\not\prec\Phi(l_n).
\end{equation}

Next we claim that $M(l):=\{t\in[0,1]:f(l(t))=\Phi(l)\}\neq\emptyset$ for all $l\in\Gamma$. In fact, by $f\in HC^1(X,\mathcal{I})$ and $l\in C([0,1],X)$, we know that $\textrm{Sup}_{x\in[0,1]}f(l(t))$ exists and so there exists a sequence $\{t_n\}$ in $[0,1]$ such that $f(l(t_n))\rightarrow\textrm{Sup}_{x\in[0,1]}f(l(t))$. Moreover, there exits a subsequence $\{t_{n_k}\}$ of $\{t_n\}$ such that $t_{n_k}\rightarrow t_0\in[0,1]$. Thus, we have $f(l(t_0))=\textrm{Sup}_{x\in[0,1]}f(l(t))$ due to the Hausdorff continuity of $f$ and continuity of $l$. This shows that $M(l)\neq\emptyset$.
Obviously, for each $l\in\Gamma$, $M(l)$ is closed and so it is compact. We show that $M(l)\subset(0,1)$. Indeed, let $t_0\in M(l)\cap\{0,1\}$. Then  $\alpha\preccurlyeq\Phi(l)=f(l(t_0))$. On the other hand, by condition (i), we obtain $f(l(t_0))\prec\alpha$, which is a contradiction.

Let $\Gamma_0=\{\phi\in C([0,1],X):\phi(0)=\theta,\phi(1)=\theta\}$, where $\theta$ is the null element in $C([0,1],X)$. Clearly, $\Gamma_0$ is a closed subspace of $ C([0,1],X)$. Let $\|\cdot\|_{\Gamma_0}$ be a norm of $\Gamma_0$ defined by
$$\|\phi\|_{\Gamma_0}=\max_{t\in[0,1]}\|\phi(t)\|_X, \quad \forall \phi\in\Gamma_0.$$
Then, for each $h\in\Gamma_0$ with $\|h\|_{\Gamma_0}=1$, $t_j>0$ with $t_j\rightarrow 0$, and $\xi_j\in M(l_n+t_jh)$, $j=1,2,\cdots$, it follows from (\ref{eq4.10}) and Remark \ref{re2.5} that
\begin{equation}\label{eq4.11}
\frac{1}{t_j}(f(l_n(\xi_j)+t_jh(\xi_j))\ominus_{gH} f(l_n(\xi_j)))\not\prec\left[-\frac{1}{n},-\frac{1}{n}\right].
\end{equation}
Thanks to $\xi_j\in(0,1)$, without loss of generality, we can assume that the sequence $\{\xi_j\}$ converges to $\eta_n\in[0,1]$, which depends on $\l_n$, $t_j$ and $h$. By Proposition \ref{p2.2} and $f\in HC^1(X,\mathcal{I})$, taking the limit as $j\rightarrow+\infty$ in (\ref{eq4.11}), one has
\begin{equation}\label{eq4.12}
f'(l_n(\eta_n))(h(\eta_n))\not\prec\left[-\frac{1}{n},-\frac{1}{n}\right].
\end{equation}
Now we claim that there exists $\eta_n^*\in M(l_n)$ such that, for any $\eta\in X$ with $\|\eta\|_X=1$,
\begin{equation}\label{eq4.13}
f'(l_n(\eta_n^*))(\eta)\not\prec\left[-\frac{1}{n},-\frac{1}{n}\right].
\end{equation}
If not, then for each $\eta\in M(l_n)$, there exists $y_\eta\in X$ with $\|y_\eta\|_X=1$ such that
$$
f'(l_n(\eta))(y_\eta)\prec\left[-\frac{1}{n},-\frac{1}{n}\right].
$$
Thus, by the Hausdorff continuity of $f'(\cdot)(y_\eta)$ and the definition of $\prec$, there exists a neighborhood $U_\eta$ of $\eta$ in $(0,1)$ such that, for any $\xi\in U_\eta$, one has $f'(l_n(\xi))(y_\eta)\prec[-\frac{1}{n},-\frac{1}{n}]$. Because $M(l_n)$ is compact, there are finite neighborhoods $\{U_{\eta_i}\}_{i=1}^m$ such that $M(l_n)\subseteq\cup_{i=1}^mU_{\eta_i}$. Thus, we can obtain $\{y_{\eta_i}\}_{i=1}^m$ with $\|y_{\eta_i}\|_X=1$ satisfying
$$
f'(l_n(\xi))(y_{\eta_i})\prec\left[-\frac{1}{n},-\frac{1}{n}\right],\quad\forall \xi\in U_{\eta_i}.
$$
Consider a partition of unity associated with $\{U_{\eta_i}\}_{i=1}^m$ such that $\rho_i:[0,1]\to[0,1]$ is continuous, $\textrm{supp}(\rho_i):=cl\{\xi\in [0,1]:\rho_i(\xi)\neq0\}\subseteq U_{\eta_i}$ for $i=1,2,\cdots,m$, and
$\Sigma_{i=1}^m\rho_i(\xi)\equiv1$  for all $\xi\in M(l_n)$.
Let $y(\xi)=\sum_{i=1}^m\rho_i(\xi)y_{\eta_i}$. Then $y$ is continuous on $[0,1]$ with $y(0)=\theta$ and $y(1)=\theta$. Thus, $y\in \Gamma_0$ with $\|y\|_{\Gamma_0}\leq1$. On the other hand,  we can choose finite neighborhoods $\{U_{\eta_i}\}_{i=1}^m$ such that, for each $\xi^*\in M(l_n)$, there is only one $i_0\in\{1,2,\cdots,m\}$ satisfying $\xi^*\in U_{\eta_{i_0}}$.
Therefore, we have $\|y\|_{\Gamma_0}=1$ and $f'(l_n(\xi))(y(\xi))\prec[-\frac{1}{n},-\frac{1}{n}]$ for all $\xi\in M(l_n)$, which contradict (\ref{eq4.12}). Thus, we conclude that (\ref{eq4.13}) holds.

Finally, by setting $x_n=l_n(\eta_n^*)$ in (\ref{eq4.13}), we have $f'(x_n)(\eta)\not\prec[-\frac{1}{n},-\frac{1}{n}]$ for all $\eta\in X$ with $\|\eta\|_X=1$. Similar to the proof of Theorem \ref{th3.3}, we can show that $f'(x_n)(\eta)\rightarrow[0,0]$ or $0\in f'(x_n)(\eta)$ for $n=1,2,\cdots$. Going back to (\ref{eq4.9}), we obtain $f(x_n)\rightarrow C$. Consequently, applying the $(\textrm{Palais-Smale})_C$ condition, there exists a subsequence $\{x_{n_k}\}$ of $\{x_n\}$ converging to $x_0\in X$. Thanks to $f\in HC^1(X,\mathcal{I})$, we have $f(x_0)=C$ and $0\in f'(x_0)(\eta)$ for all $\eta\in X$ with $\|\eta\|_X=1$. It follows from the linearity of $f'(x_0)$ that $0\in f'(x_0)(\eta)$ for all $\eta\in X$. This shows that $C$ is a critical value of $f$.
\end{proof}

\begin{remark}\label{re4.4}
If $f$ is a real-valued function, then Theorem \ref{th4.5} reduces to the classical Mountain Pass Theorem obtained by Ambrosetti and Rabinowitz \cite{AR}.
\end{remark}

\subsection{Noncooperative interval-valued games}

In this subsection, we consider an approximate Nash equilibrium of a noncooperative $n$-person game in interval-valued settings.

Let $G=(\mathbf{N},\{X_i\}_{i\in \mathbf{N}},\{f_i\}_{i\in \mathbf{N}})$ be a noncooperative interval-valued game, where
\begin{itemize}
\item[(i)] $\mathbf{N}:=\{1,2,\cdots,n\}$ is the set of $n$ players;
\item[(ii)] for each $i\in\mathbf{N}$, the strategy set of the $i$th player, denoted by $X_i$, is a complete metric space with the metric $d_i$;
\item[(iii)] set $V:=\prod_{i\in\mathbf{N}}X_i$;
\item[(iv)] for each $i\in\mathbf{N}$, $f_i:V\to\mathcal{I}$ is an interval-valued loss function of the $i$th player.
\end{itemize}
Set $V_{-i}:=\prod_{j\in\mathbf{N}\backslash\{i\}}X_j$.  For each $i\in\mathbf{N}$, we define
$$x_{-i}:=(x_1,\cdots,x_{i-1},x_{i+1},\cdots,x_n)\in V_{-i}, \quad \forall x=(x_1,x_2,\cdots,x_n)\in V.$$
Moreover, let $(y_i,x_{-i}):=(x_1,\cdots,x_{i-1},y_i,x_{i+1},\cdots,x_n)\in V$  for each $i\in\mathbf{N}$ and $y_i\in X_i$.

We first introduce an approximate Nash equilibrium of $G$ as follows.
\begin{definition}\label{de4.4}
Given $\varepsilon>0$, a strategy profile $\overline{x}=(\overline{x}_1,\overline{x}_2,\cdots,\overline{x}_n)\in V$ is said to be an $\varepsilon$-Nash equilibrium of $G$ if, for each $i\in\mathbf{N}$, $f_i(y_i,\overline{x}_{-i})+[\varepsilon d_i(\overline{x}_i,y_i),\varepsilon d_i(\overline{x}_i,y_i)]\not\prec f_i(\overline{x}_i,\overline{x}_{-i})$ for all $y_i\in X_i$.
\end{definition}

We then define a metric $\widehat{d}$ on $V=\prod_{i\in\mathbf{N}}X_i$ and an interval-valued bifunction $f:V\times V\to\mathcal{I}$ by setting
$$
\widehat{d}(x,y)=\sum_{i=1}^n d_i(x_i,y_i), \quad f(x,y):=\sum_{i=1}^n (f_i(y_i,x_{-i})\ominus_{gH}f_i(x_i,x_{-i})),
$$
where $x=(x_1,x_2,\cdots,x_n)\in V$ and $y=(y_1,y_2,\cdots,y_n)\in V$. Clearly, $(V,\widehat{d})$ is a complete metric space.

Now we give the following result which characterizes an $\varepsilon$-Nash equilibrium of $G$.
\begin{proposition}\label{p4.2}
For every $\varepsilon>0$, if there is $\overline{x}=(\overline{x}_1,\overline{x}_2,\cdots,\overline{x}_n)\in V$ such that, for any $y\in V$,
$$
f(\overline{x},y)+[\varepsilon \widehat{d}(\overline{x},y),\varepsilon \widehat{d}(\overline{x},y)]\not\prec[0,0],
$$
then $\overline{x}$ is an $\varepsilon$-Nash equilibrium of $G$.
\end{proposition}
\begin{proof}
For each $i\in\mathbf{N}$, let $y=(\overline{x}_1,\cdots,\overline{x}_{i-1},y_i,\overline{x}_{i+1},\cdots,\overline{x}_n)$. By the definition of $f$, it follows from Remark \ref{re2.1} that $f(\overline{x},y)=f_i(y_i,\overline{x}_{-i})\ominus_{gH}f_i(\overline{x}_i,\overline{x}_{-i})$. Also, it  follows from the definition of $\widehat{d}$ that $\widehat{d}(\overline{x},y)=d_i(\overline{x}_i,y_i)$.
Thanks to $f(\overline{x},y)+[\varepsilon \widehat{d}(\overline{x},y),\varepsilon \widehat{d}(\overline{x},y)]\not\prec[0,0]$, we have
$$
(f_i(y_i,\overline{x}_{-i})\ominus_{gH}f_i(\overline{x}_i,\overline{x}_{-i}))+[\varepsilon d_i(\overline{x}_i,y_i),\varepsilon d_i(\overline{x}_i,y_i)]\not\prec[0,0].
$$
By Remark \ref{re2.5}, we know that $f_i(y_i,\overline{x}_{-i})+[\varepsilon d_i(\overline{x}_i,y_i),\varepsilon d_i(\overline{x}_i,y_i)]\not\prec f_i(\overline{x}_i,\overline{x}_{-i})$ and so $\overline{x}$ is an $\varepsilon$-Nash equilibrium of $G$.
\end{proof}

In order to obtain the existence of $\varepsilon$-Nash equilibrium of $G$, we first show the following lemmas.
\begin{lemma}\label{le4.1}
Let $(X,d)$ be a complete metric space and $\varphi: X\to \mathcal{I}$ be an interval-valued function. Then for any given $A\in\mathcal{I}$, $\varphi(\cdot)\ominus_{gH}A$ is $\preccurlyeq$-lower semicontinuous providing $\varphi$ is $\preccurlyeq$-lower semicontinuous.
\end{lemma}
\begin{proof}
Let $A=[\underline{a},\overline{a}]$ and $\varphi(x)=[\underline{\varphi}(x),\overline{\varphi}(x)]$ for all $x\in X$. Then
$$
\varphi(x)\ominus_{gH}A=[\min\{\underline{\varphi}(x)-\underline{a},\overline{\varphi}(x)-\overline{a}\},\max\{\underline{\varphi}(x)-\underline{a},
\overline{\varphi}(x)-\overline{a}\}],  \quad \forall x\in X.
$$
Thanks to the $\preccurlyeq$-lower semicontinuity of $\varphi$, it follows from Proposition \ref{p2.3} that $\underline{\varphi}$ and $\overline{\varphi}$ are lower semicontinuous on $X$ and so do $\underline{\varphi}(x)-\underline{a}$ and $\overline{\varphi}(x)-\overline{a}$. Now we define $\varphi_1(x):=\min\{\underline{\varphi}(x)-\underline{a},\overline{\varphi}(x)-\overline{a}\}$
and $\varphi_2(x):=\max\{\underline{\varphi}(x)-\underline{a},\overline{\varphi}(x)-\overline{a}\}$ for all $x\in X$. Again by Proposition \ref{p2.3}, it remains to prove that $\varphi_1$ and $\varphi_2$ are lower semicontinuous on $X$. For any $r\in\mathbb{R}$, let $\{x_n\}$ be a sequence in $\{x\in X:\varphi_2(x)\leq r\}$ converging to $x_0$. Then $\underline{\varphi}(x_n)-\underline{a}\leq r$ and
$\overline{\varphi}(x_n)-\overline{a}\leq r$. We can see that $\underline{\varphi}(x_0)-\underline{a}\leq r$ and
$\overline{\varphi}(x_0)-\overline{a}\leq r$. Thus, $\varphi_2(x_0)\leq r$  and so $\varphi_2$ is lower semicontinuous on $X$.
Next, for any $s\in\mathbb{R}$, let $\{y_n\}$ be a sequence in $\{x\in X:\varphi_1(x)\leq s\}$ converging to $y_0$. If $y_0\not\in\{x\in X:\varphi_1(x)\leq s\}$, then $\varphi_1(y_0)> s$. This implies that $\underline{\varphi}(y_0)-\underline{a}> s$ and
$\overline{\varphi}(y_0)-\overline{a}>s$. Because $\underline{\varphi}(x)-\underline{a}$ and $\overline{\varphi}(x)-\overline{a}$ are lower semicontinuous for all $x\in X$, we have $s<\underline{\varphi}(y_0)-\underline{a}\leq\liminf_{n\rightarrow+\infty}(\underline{\varphi}(y_n)-\underline{a})$ and $s<\overline{\varphi}(y_0)-\overline{a}\leq\liminf_{n\rightarrow+\infty}(\overline{\varphi}(y_n)-\overline{a})$. It follows that there exist a subsequence $\{z_k\}$ of $\{y_n\}$ and an integer $N_1>0$ such that, for each $k>N_1$, $\underline{\varphi}(z_k)-\underline{a}>s$. Moreover, there exist a subsequence $\{z_{n_k}\}$ of $\{z_k\}$ and an integer $N_2>0$ such that, for each $n_k>N_2$, $\overline{\varphi}(z_{n_k})-\overline{a}>s$. Let $N:=\max\{N_1, N_2\}$. Then for any $n_k>N$, one has $\underline{\varphi}(z_{n_k})-\underline{a}>s$ and $\overline{\varphi}(z_{n_k})-\overline{a}>s$, which contradict the fact that $\varphi_1(y_n)\leq s$. Consequently, we have $y_0\in\{x\in X:\varphi_1(x)\leq s\}$ and so $\varphi_1$ is lower semicontinuous on $X$.
\end{proof}

\begin{lemma}\label{le4.2}
Let $(X,d)$ be a complete metric space and $\varphi: X\to \mathcal{I}$ be an interval-valued function. Then for any given $A\in\mathcal{I}$,  $\varphi(\cdot)\ominus_{gH}A$ is $\preccurlyeq$-lower bounded providing $\varphi$ is $\preccurlyeq$-lower bounded.
\end{lemma}
\begin{proof}
Let $A=[\underline{a},\overline{a}]$ and $\varphi(x)=[\underline{\varphi}(x),\overline{\varphi}(x)]$ for all $x\in X$. Then
$$
\varphi(x)\ominus_{gH}A=[\min\{\underline{\varphi}(x)-\underline{a},\overline{\varphi}(x)-\overline{a}\},\max\{\underline{\varphi}(x)-\underline{a},
\overline{\varphi}(x)-\overline{a}\}], \quad \forall x\in X.
$$
Because $\varphi$ is $\preccurlyeq$-lower bounded, there exists $M=[\underline{m},\overline{m}]\in\mathcal{I}$ such that $M\preccurlyeq \varphi(x)$ for all $x\in X$. Thus, $\underline{m}\leq \underline{\varphi}(x)$ and $\overline{m}\leq \overline{\varphi}(x)$ for all $x\in X$.
Moreover, we have $\underline{m}-\underline{a}\leq \underline{\varphi}(x)-\underline{a}$ and $\overline{m}-\overline{a}\leq \overline{\varphi}(x)-\overline{a}$ for all $x\in X$. It follows from Remark \ref{re2.6} that $\varphi(\cdot)\ominus_{gH}A$ is $\preccurlyeq$-lower bounded on $X$.
\end{proof}

We end this subsection by giving the following theorem which ensures the existence of approximate Nash equilibria of $G$.
\begin{theorem}\label{th4.6}
Suppose that
\begin{itemize}
\item[(i)] for each $i\in\mathbf{N}$ and $x_{-i}\in V_{-i}$, $f_i(\cdot,x_{-i})$ is $\preccurlyeq$-lower semicontinuous and $\preccurlyeq$-lower bounded on $X_i$;
\item[(ii)] for any $x,y,z\in V$, $f(x,z)\preccurlyeq f(x,y)+f(y,z)$.
\end{itemize}
Then for any given $\varepsilon>0$,  $G$ admits an $\varepsilon$-Nash equilibrium.
\end{theorem}
\begin{proof}
Thanks to the fact that $f(x,y)=\sum_{i=1}^n (f_i(y_i,x_{-i})\ominus_{gH}f_i(x_i,x_{-i}))$, it follows from Lemma \ref{le4.1} and Remark \ref{re2.7} that, for each $x\in V$, $f(x,\cdot)$ is $\preccurlyeq$-lower semicontinuous on $V$. By Lemma \ref{le4.2}, we can deduce that, for each $x\in V$, $f(x,\cdot)$ is $\preccurlyeq$-lower bounded on $V$. Thus, applying Theorem \ref{th3.4} and Remark \ref{re3.7}, we know that for any given $\varepsilon>0$, there exists $\overline{x}\in V$ such that, for any $y\in V$, $f(\overline{x},y)+[\varepsilon d(\overline{x},y),\varepsilon d(\overline{x},y)]\not\prec[0,0]$. It follows from Proposition \ref{p4.2} that $\overline{x}$ is an $\varepsilon$-Nash equilibrium of $G$.
\end{proof}

\subsection{Interval-valued optimal control problems}

In this subsection, we give some existence theorems for approximate minimal solutions of interval-valued optimal control problems driven by interval-valued differential equations.

We first present the generalized differentiability for interval-valued functions.
\begin{definition}(\cite{CR})\label{de4.5}
Let $(a,b)$ be an open interval in $\mathbb{R}$. An interval-valued function $f:(a,b) \to \mathcal{I}$ is said to be generalized differentiable at $t_0\in (a,b)$ if there exists $\frac{df}{dt}(t_0)\in \mathcal{I}$ such that either
\begin{itemize}
\item[(i)] for any $h>0$ sufficiently near to 0, the Hukuhara differences $f(t_0+h)\ominus_H f(t_0)$ and $f(t_0)\ominus_H f(t_0-h)$ exist, and
$$
\frac{df}{dt}(t_0):=\lim_{h\rightarrow0_+}\frac{f(t_0+h)\ominus_{H}f(t_0)}{h}=\lim_{h\rightarrow0_+}\frac{f(t_0)\ominus_{H}f(t_0-h)}{h}
$$
or
\item[(ii)] for any $h<0$ sufficiently near to 0, the Hukuhara differences $f(t_0+h)\ominus_H f(t_0)$ and $f(t_0)\ominus_H f(t_0-h)$ exist, and
$$
\frac{df}{dt}(t_0):=\lim_{h\rightarrow0_-}\frac{f(t_0+h)\ominus_{H}f(t_0)}{h}=\lim_{h\rightarrow0_-}\frac{f(t_0)\ominus_{H}f(t_0-h)}{h}.
$$
\end{itemize}
\end{definition}

\begin{definition}\label{de4.6}
An interval-valued function $f:(a,b) \to \mathcal{I}$ is said to be (i)-differentiable (resp., (ii)-differentiable) on $(a,b)$ if $f$ is generalized differentiable at every point $t_0\in (a,b)$ in the sense of (i) (resp., (ii)) in Definition \ref{de4.5}.
\end{definition}

\begin{remark}\label{re4.5}
From \cite{CR,SB}, we can see that (I) if $f$ is generalized differentiable, then $f$ is Hausdorff continuous; (II) if $f:(a,b)\to \mathcal{I}$ and $g:(a,b)\to \mathcal{I}$ are (i)-differentiable (resp., (ii)-differentiable), then $f+g$ is also (i)-differentiable (resp., (ii)-differentiable); (III) for any $t\in (a,b)$ and $\lambda\in\mathbb{R}$, $\frac{d(f+g)}{dt}(t)=\frac{df}{dt}(t)+ \frac{dg}{dt}(t)$, $\frac{d(\lambda f)}{dt}(t)=\lambda \frac{df}{dt}(t)$.
\end{remark}

From \cite{AC,RCL},  an interval-valued function $f:[a,b]\to\mathcal{I}$ is called integrably bounded on $[a,b]$ if there exists an Lebesgue integrable function $h: [a,b]\to[0,+\infty)$ such that, for any $t\in [a,b]$ and $x\in f(t)$, $|x|\leq h(t)$. Moreover, an interval-valued function $f$ is said to be measurable on $[a,b]$ if, for any closed subset $A\subset \mathbb{R}$, $\{t\in[a,b]: f(t)\cap A\neq \emptyset\}$ is Lebesgue measurable.

We recall the Aumann integral of an interval-valued function defined as follows.
\begin{definition}\label{de4.7} (\cite{AC})
Let $f:[a,b]\to \mathcal{I}$ be an interval-valued function. The Aumann integral of $f$ over $[a,b]$ is given by
$$
 \int_a^b f(t)dt=\left\{(L)\int_a^b g(t)dt: g\in S(f)\right\},
$$
where $S(f):=\{g:[a,b]\to\mathbb{R}: g ~\textrm{Lebesgue integrable}, ~g(t)\in f(t),\;\forall t\in[a,b]\}$ and $(L)\int_a^b g(t)dt$ denotes the Lebesgue integral of $g$ over $[a,b]$. Moreover, $f$ is said to be Aumann integrable if $S(f)\neq\emptyset$.
\end{definition}
\begin{remark}\label{re4.6}
If $f:[a,b]\to \mathcal{I}$ is measurable and integrably bounded with $f(t)=[\underline{f}(t),\overline{f}(t)]$, then it follows from \cite{RCL} that $f$ is Aumann integrable, $\underline{f}$ and $\overline{f}$ are Lebesgue integrable and $\int_a^b f(s)ds=[(L)\int_a^b \underline{f}(s)ds,(L)\int_a^b \overline{f}(s)ds]$. Moreover,  if $f:[a,b]\to \mathcal{I}$ is Hausdorff continuous, then it is Aumann integrable.
\end{remark}

We also need the following lemmas.
\begin{lemma}( \cite{CR,SB,KNR1})\label{le4.3}
Let $f:[a,b]\to\mathcal{I}$ be an interval-valued function. Then the following conclusions are true:
\begin{itemize}
\item [(i)] if $f$ is Hausdorff continuous,  then for any $t\in [a,b]$, $F(t):=\int_a^t f(s)ds$ is (i)-differentiable and $\frac{dF}{dt}(t)=f(t)$;

\item [(ii)] if $f$ is (i)-differentiable and $\frac{df}{dt}$ is Aumann integrable, then $f(t)=f(a)+\int_a^t \frac{df}{ds}(s)ds$ for all $t\in [a,b]$;

\item[(iii)] ff $f$ is Hausdorff continuous, then for any $t\in [a,b]$, $G(t):=\chi\ominus_H\int_a^t -f(s)ds$ is (ii)-differentiable and $\frac{dG}{dt}(t)=f(t)$, where $\chi\in \mathcal{I}$ is such that the previous Hukuhara difference exists;

\item[(iv)] if $f$ is (ii)-differentiable and $\frac{df}{dt}$ is Aumann integrable, then $f(t)=f(a)\ominus_{H}\int_a^t -\frac{df}{ds}(s)ds$ for all $t\in [a,b]$.
\end{itemize}
\end{lemma}

\begin{lemma}(\cite{AC,RCL,Kaleva})\label{le4.4}
Let $f: [a,b]\to \mathcal{I}$  and $g: [a,b]\to \mathcal{I}$ be Aumann integrable. Then
\begin{itemize}
\item[(i)] $\int_a^b f(s)+g(s)ds=\int_a^b f(s)ds +\int_a^b g(s)ds$;
\item[(ii)] for any $\lambda\in\mathbb{R}$, $\int_a^b \lambda f(s)ds=\lambda \int_a^b f(s)ds$;
\item[(iii)] $d_H(f(\cdot),g(\cdot))$ is Lebesgue integrable on $[a,b]$;
\item[(iv)] $d_H(\int_a^b f(s)ds, \int_a^b g(s)ds)\leq (L)\int_a^b d_H(f(s),g(s))ds$;
\item[(v)] $\int_a^b f(s)ds=\int_a^c f(s)ds+ \int_c^b f(s)ds$, where $c\in[a,b]$.
\end{itemize}
\end{lemma}

\begin{lemma}(\cite{KNR2})\label{le4.5}
Let $A,B,C,D\in \mathcal{I}$. If the Hukuhara differences $A\ominus_H C$ and $B\ominus_H D$ exist, then
$d_H(A\ominus_H C, B\ominus_H D)\leq d_H(A,B)+d_H(C,D)$.
\end{lemma}

\begin{lemma}\label{le4.6}
Let $A,B,C\in \mathcal{I}$. Then the following statements are true:
\begin{itemize}
\item[(i)] if the Hukuhara differences $A\ominus_H B$ and $B\ominus_H C$ exist, then $(A\ominus_H B)+ C=A\ominus_H(B\ominus_H C)$;
\item[(ii)] if the Hukuhara differences $A\ominus_H B$ and $(A\ominus_H B)\ominus_H C$ exist, then $(A\ominus_H B)\ominus_H C=A\ominus_H (B+C)$.
\end{itemize}
\end{lemma}
\begin{proof}
(i) Let $A\ominus_H B=D_1$ and $B\ominus_H C=D_2$, where $D_1,D_2\in \mathcal{I}$. Then we have $A=B+ D_1$ and $B=C+D_2$.  It follows that $A=D_2+ D_1+C$ and so $A\ominus_H D_2=D_1+C$. Thus, we can see that $A\ominus_H(B\ominus_H C)=(A\ominus_H B)+C$.

(ii) Let $A\ominus_H B=A_1$ and $(A\ominus_H B)\ominus_H C=A_2$, where $A_1,A_2\in \mathcal{I}$. Then $A=B+A_1$ and $A_1=C+A_2$ and so
$A=B+C+A_2$. Thus, $A\ominus_H (B+C)=A_2=(A\ominus_H B)\ominus_H C$.
\end{proof}

Let $T>0$ and $C([0,T],\mathcal{I})$ be the set of all Hausdorff continuous interval-valued functions from $[0,T]$ to $\mathcal{I}$ with the metric $D$, where $D(x,y):=\sup_{t\in [0,T]}d_H(x(t),y(t))$ for all $x,y\in C([0,T],\mathcal{I})$. Clearly, $(C([0,T],\mathcal{I}),D)$ is a complete metric space.

We consider a system governed by the following interval-valued control differential equation:
\begin{eqnarray}\label{eq4.14}
 \left\{
\begin{array}{ll}
\displaystyle\frac{dx}{dt}(t)=f(t,x(t),u(t)), &\mbox {$t\in[0,T],$}\\
x(0)=x_0\in \mathcal{I},
\end{array}
\right.
\end{eqnarray}
where $x:[0,T]\to \mathcal{I}$ is the state of the system, $u:[0,T]\to\mathbb{R}^n$ is the control, $f:[0,T]\times \mathcal{I}\times \mathbb{R}^n\to \mathcal{I}$ is Hausdorff continuous, and $x$ is (i)- or (ii)-differentiable on $[0,T]$.

In order to obtain our results, we also need the following assumptions:
\begin{itemize}
\item[(A1)]  $u\in U_{ad}$, where $U_{ad}$ is the set of all continuous functions from $[0,T]$ to $\mathbb{R}^n$ with the norm $\|u\|_{U_{ad}}=\max_{t\in[0,T]}\|u(t)\|_{\mathbb{R}^n}$;
\item[(A2)] there exist $k_1>0$ and $k_2>0$ such that, for any $t\in[0,T]$, $x_1,x_2\in \mathcal{I}$ and $u_1,u_2\in\mathbb{R}^n$, $$d_H(f(t,x_1,u_1),f(t,x_2,u_2))\leq k_1d_H(x_1,x_2)+k_2\|u_1-u_2\|_{\mathbb{R}^n}.$$
\end{itemize}

\begin{remark}\label{re4.7}
For any given $u\in U_{ad}$, if the (i)-differentiability is taken on equation (\ref{eq4.14}), then by Theorem 6.1 in \cite{Kaleva}, assumption (A2) implies that there exists a unique solution $x$ of equation (\ref{eq4.14}) which can be expressed as
$x(t)=x_0+\int_{0}^t f(s,x(s),u(s))ds$ for all $t\in [0,T]$.
\end{remark}

Clearly, $U_{ad}$ is a Banach space and $f(\cdot,x(\cdot),u(\cdot))$ is Hausdorff continuous on $[0,T]$ when $x\in C([0,T],\mathcal{I})$ and $u\in U_{ad}$.

The following theorem ensures the existence of a unique solution if the (ii)-differentiability is taken on equation (\ref{eq4.14}).
\begin{theorem}\label{th4.7}
Let $u\in U_{ad}$ be given and the (ii)-differentiability be taken on equation (\ref{eq4.14}). Suppose that (A1) and (A2) are satisfied. If for any $x\in C([0,T],\mathcal{I})$ and $\chi\in \mathcal{I}$, the Hukuhara difference
$\chi\ominus_H\int_0^t -f(s,x(s),u(s))ds$ exists for all $t\in[0,T]$, then there exists a unique solution $x$ of equation (\ref{eq4.14}) which can be expressed as $x(t)=x_0\ominus_H\int_0^t-f(s,x(s),u(s))ds$.
\end{theorem}
\begin{proof}
Let $u\in U_{ad}$ be given. Then by Remark \ref{re4.5} and Lemma \ref{le4.3}, we know that $x$ is a solution of equation (\ref{eq4.14}) if and only if
$x(t)=x_0\ominus_H\int_0^t-f(s,x(s),u(s))ds$ for all $t\in[0,T]$.

For any given $(t_1,y)\in [0,T]\times \mathcal{I}$ and $\delta>0$ satisfying $\delta k_1<1$, we claim that the following interval-valued differential equation
\begin{eqnarray}\label{eq4.15}
 \left\{
\begin{array}{ll}
\displaystyle\frac{dx}{dt}(t)=f(t,x(t),u(t)),\\
x(t_1)=y
\end{array}
\right.
\end{eqnarray}
has a unique solution on $[t_1,t_1+\delta]$. In fact, for any $x\in C([t_1,t_1+\delta],\mathcal{I})$, we can define
$$
(Sx)(t):=y\ominus_H\int_{t_1}^t -f(s,x(s),u(s))ds,\;\forall t\in[t_1,t_1+\delta].
$$
By the hypothesis, we know that, for any $t\in[t_1,t_1+\delta]$, $y\ominus_H\int_{0}^t-f(s,x(s),u(s))ds$ exists.
It follows from Lemma \ref{le4.4} that
$$\int_{0}^t -f(s,x(s),u(s))ds \ominus_H\int_{0}^{t_1}-f(s,x(s),u(s))ds=\int_{t_1}^t-f(s,x(s),u(s))ds.$$
Moreover, invoking Lemma \ref{le4.6}(i),  for any $t\in[t_1,t_1+\delta]$, one has
$$
(y\ominus_H\int_{0}^t -f(s,x(s),u(s))ds)+\int_{0}^{t_1} -f(s,x(s),u(s))ds=y\ominus_H\int_{t_1}^t-f(s,x(s),u(s))ds
$$
and so  $y\ominus_H\int_{t_1}^t -f(s,x(s),u(s))ds$ exists. Thus, $(Sx)(t)$ is well-defined. Taking into account Lemma \ref{le4.3} and Remark \ref{re4.5}, we have $Sx\in C([t_1,t_1+\delta],\mathcal{I})$. Moreover, by assumption (A2) and  Lemmas \ref{le4.4} and \ref{le4.5},  for any $z_1,z_2\in C([t_1,t_1+\delta],\mathcal{I})$,
\begin{align}
D(S(z_1),S(z_2))&=\sup_{t\in[t_1,t_1+\delta]} d_H(y\ominus_H\int_{t_1}^t-f(s,z_1(s),u(s))ds,y
 \ominus_H\int_{t_1}^t-f(s,z_2(s),u(s))ds)\nonumber\\
&\leq\sup_{t\in[t_1,t_1+\delta]}d_H(\int_{t_1}^t f(s,z_1(s),u(s))ds,\int_{t_1}^t f(s,z_2(s),u(s))ds)\nonumber\\
&\leq (L)\int_{t_1}^{t_1+\delta}d_H(f(s,z_1(s),u(s)),f(s,z_2(s),u(s)))ds\nonumber\\
&\leq (L)\int_{t_1}^{t_1+\delta} k_1 d_H(z_1(s),z_2(s))ds\nonumber\\
&\leq \delta k_1 D(z_1,z_2).\nonumber
\end{align}
Since $\delta k_1<1$,  Banach's fixed point theorem shows that there exists a unique point $x\in C([t_1,t_1+\delta],\mathcal{I})$ such that
$$
x(t)=y\ominus_H\int_{t_1}^t -f(s,x(s),u(s))ds,\quad \forall t\in [t_1,t_1+\delta].
$$
Thus, $x$ is the unique solution on $[t_1,t_1+\delta]$ for equation (\ref{eq4.15}).

Since $[0,T]$ can be expressed as the union of a finite family of intervals $I_k$ ($k=1,2,\cdots$) with the length of each
interval less than $\delta$, there is a unique solution on each interval $I_k$ for equation (\ref{eq4.15}). Thus, by Lemma \ref{le4.4} and Lemma \ref{le4.6} (ii), piecing these solutions together, we can obtain a unique solution $x$ on $[0,T]$ for equation (\ref{eq4.14}) with $x(t)=x_0\ominus_H\int_0^t -f(s,x(s),u(s))ds$ for all $t\in[0,T]$.
\end{proof}

Next we consider the following interval-valued optimal control problem (IOCP):
$$
\textrm{Minimize}\; F(u):=\int_0^T L(t,x(t),u(t))dt,\;\textrm{subject to}\; u\in U_{ad},
$$
where $L:[0,T]\times \mathcal{I}\times \mathbb{R}^n\to \mathcal{I}$ with $L(t,x(t),u(t))=[\underline{L}(t,x(t),u(t)),\overline{L}(t,x(t),u(t))]$ and $x$ is a solution for equation (\ref{eq4.14}) corresponding to $u\in U_{ad}$ such that $L(\cdot,x(\cdot),u(\cdot))$ is measurable and integrably bounded on $[0,T]$.

Now we introduce an approximate minimal solution of (IOCP).
\begin{definition}\label{de4.8}
Given $\varepsilon>0$, a point $u_0\in U_{ad}$ is said to be an $\varepsilon$-minimal solution of (IOCP) if, for any $u\in U_{ad}$, $F(u)+[\varepsilon\|u-u_0\|_{U_{ad}},\varepsilon\|u-u_0\|_{U_{ad}}]\not\prec F(u_0)$.
\end{definition}
\begin{remark}\label{re4.8}
We would like to mention that a minimal solution in the sense of Definition \ref{de2.7} implies an $\varepsilon$-minimal solution.
\end{remark}

In order to obtain approximate minimal solutions of (IOCP), we also need the following lemmas.

\begin{lemma}\label{le4.7}(\cite{Bellman})
Let $g,h:[0,+\infty)\to [0,+\infty)$ be continuously real-valued functions and $c$ be a constant with $c\geq0$. If for any $t\in[0,+\infty)$,
$g(t)\leq c+(L)\int_0^tg(s)h(s)ds$, then $g(t)\leq ce^{(L)\int_0^t h(s)ds}$.
\end{lemma}

\begin{lemma}\label{le4.8}
Let $x,y\in C([0,T],\mathcal{I})$ and $\phi:[0,T]\to[0,+\infty)$ be a real-valued function defined by $\phi(t):=d_H(x(t),y(t))$ for all $t\in[0,T]$. Then $\phi$ is continuous on $[0,T]$.
\end{lemma}
\begin{proof}
Let $\{t_n\}$ be a sequence in $[0,T]$ converging to $t_0\in[0,T]$. Because $x$ and $y$ are Hausdorff continuous on $[0,T]$, one has
$d_H(x(t_n),x(t_0))\rightarrow0$ and $d_H(y(t_n),y(t_0))\rightarrow0$.
Then we have
\begin{align}
|\phi(t_n)-\phi(t_0)|&=|d_H(x(t_n),y(t_n))-d_H(x(t_0),y(t_0))|\nonumber\\
&\leq |d_H(x(t_n),x(t_0))+d_H(x(t_0),y(t_n))-d_H(x(t_0),y(t_0))|\nonumber\\
&\leq d_H(x(t_n),x(t_0))+d_H(y(t_n),y(t_0))\rightarrow0.\nonumber
\end{align}
This shows that $\phi$ is continuous.
\end{proof}
\begin{remark}\label{re4.9}
Note that $\phi$ is uniformly continuous on $[0,T]$ since $[0,T]$ is a compact set.
\end{remark}

We end this subsection by giving the following theorems which ensure the existence of approximate minimal solutions of (ICOP) under the generalized differentiability.
To this end, we need the following assumptions.
\begin{itemize}
\item[(A3)] $L$ is $\preccurlyeq$-lower semicontinuous on $[0,T]\times \mathcal{I}\times \mathbb{R}^n$;
\item[(A4)] for any $(t,\zeta,\eta)\in [0,T]\times \mathcal{I}\times \mathbb{R}^n$, $[0,0]\preccurlyeq L(t,\zeta,\eta)$.
\end{itemize}
\begin{remark}\label{re4.10}
By Remark \ref{re4.6}, we can see that $L(\cdot,x(\cdot),u(\cdot))$ is Aumann integrable on $[0,T]$, where  $x$ is a solution for equation (\ref{eq4.14}) corresponding to $u\in U_{ad}$. Moreover, $\underline{L}(\cdot,x(\cdot),u(\cdot))$ and $\overline{L}(\cdot,x(\cdot),u(\cdot))$ are Lebesgue integrable on $[0,T]$ and
$$\int_0^T L(t,x(t),u(t))dt=\left[(L)\int_0^T\underline{L}(t,x(t),u(t))dt,(L)\int_0^T\overline{L}(t,x(t),u(t))dt\right].$$
\end{remark}

\begin{theorem}\label{th4.8}
Let the (i)-differentiability be taken on equation (\ref{eq4.14}). Assume that (A1)-(A4) are satisfied.  Then for every given $\varepsilon>0$,
there exists an $\varepsilon$-minimal solution of (IOCP).
\end{theorem}
\begin{proof}
By assumption (A4), we have $[0,0]\preccurlyeq F(u)$ for all $u\in U_{ad}$ and so $F$ is $\preccurlyeq$-lower bounded.
We claim that $F$ is $\preccurlyeq$-lower semicontinuous on $U_{ad}$. In fact, let $\{u_n\}$ be a sequence in $U_{ad}$ converging to $ \overline{u}\in U_{ad}$. Then there exist solutions $\{x_n\}$ and $\overline{x}$ of equation (\ref{eq4.14}) associated with  $\{u_n\}$ and $\overline{u}$, respectively.
Thus,
\begin{equation}\label{eq4.16}
x_n(t)=x_0+\int_0^tf(s,x_n(s),u_n(s))ds ,\quad \overline{x}=x_0+\int_0^tf(s,\overline{x}(s),\overline{u}(s))ds,\quad \forall t\in[0,T].
\end{equation}
Now we want to show that $D(x_n, \overline{x})\rightarrow0$.
Because $u_n$ converges to $\overline{u}$, for any $\delta>0$, there exists $N_1>0$ such that, for each $n>N_1$, $\|u_n(t)-\overline{u}(t)\|_{\mathbb{R}^n}<\delta$ for all $t\in[0,T]$. Moreover, by assumption (A2) and Lemma \ref{le4.4}, it follows from (\ref{eq4.16}) that, for any $n>N_1$ and $t\in[0,T]$,
\begin{align}
d_H(x_n(t),\overline{x}(t))
&\leq (L)\int_{0}^{t}d_H(f(s,x_n(s),u_n(s)),f(s,\overline{x}(s),\overline{u}(s)))ds\nonumber\\
&\leq (L)\int_{0}^{t}k_1 d_H(x_n(s),\overline{x}(s))+k_2\|u_n(s)-\overline{u}(s)\|_{\mathbb{R}^n}ds\nonumber\\
&\leq (L)\int_{0}^{t}k_1 d_H(x_n(s),\overline{x}(s))ds+(L)\int_{0}^{t}\delta k_2 ds\nonumber\\
&\leq \delta k_2T+(L)\int_{0}^{t}k_1 d_H(x_n(s),\overline{x}(s))ds.\nonumber
\end{align}
According to Lemmas \ref{le4.7} and \ref{le4.8}, we know that, for any $n>N_1$ and $t\in[0,T]$,
\begin{equation}\label{eq4.17}
d_H(x_n(t),\overline{x}(t))\leq\delta k_2T e^{(L)\int_0^t k_1ds}\leq\delta k_2Te^{k_1T}.
\end{equation}
Since $\delta>0$ is arbitrary, inequality (\ref{eq4.17}) implies $\sup_{t\in[0,T]}d_H(x_n(t),\overline{x}(t))\rightarrow0$. In other words, $D(x_n,\overline{x})\rightarrow0$.
By setting $F(u)=[\underline{F}(u),\overline{F}(u)]$, it follows from Remark \ref{re4.10} that
$$
F(u)=[\underline{F}(u),\overline{F}(u)]=\left[(L)\int_0^T\underline{L}(t,x(t),u(t))dt,(L)\int_0^T\overline{L}(t,x(t),u(t))dt\right].
$$
Thanks to the facts that $\underline{L}(\cdot,\cdot,\cdot)\geq0$ and $\overline{L}(\cdot,\cdot,\cdot)\geq0$, applying Fatou's lemma, we can see that
$$
(L)\int_0^T\liminf_{n\rightarrow+\infty}\underline{L}(t,x_n(t),u_n(t))dt\leq\liminf_{n\rightarrow+\infty}~(L)\int_0^T\underline{L}(t,x_n(t),u_n(t))dt
$$
and
$$
(L)\int_0^T\liminf_{n\rightarrow+\infty}\overline{L}(t,x_n(t),u_n(t))dt\leq\liminf_{n\rightarrow+\infty}~(L)\int_0^T\overline{L}(t,x_n(t),u_n(t))dt.
$$
Because $L$ is $\preccurlyeq$-lower semicontinuous, it follows from Proposition \ref{p2.3} that, for any $t\in[0,T]$,
$$ \underline{L}(t,\overline{x}(t),\overline{u}(t))\leq\liminf_{n\rightarrow+\infty}\underline{L}(t,x_n(t),u_n(t)),\quad
\overline{L}(t,\overline{x}(t),\overline{u}(t))\leq\liminf_{n\rightarrow+\infty}\overline{L}(t,x_n(t),u_n(t)).$$
Consequently, we have
$$
(L)\int_0^T\underline{L}(t,\overline{x}(t),\overline{u}(t))dt\leq\liminf_{n\rightarrow+\infty}~(L)\int_0^T\underline{L}(t,x_n(t),u_n(t))dt
$$
and
$$
(L)\int_0^T\overline{L}(t,\overline{x}(t),\overline{u}(t))dt\leq\liminf_{n\rightarrow+\infty}~(L)\int_0^T\overline{L}(t,x_n(t),u_n(t))dt.
$$
It follows that $\underline{F}(u)=\int_0^T\underline{L}(t,x(t),u(t))dt$ and $\overline{F}(u)=\int_0^T\overline{L}(t,x(t),u(t))dt$ are lower semicontinuous for all $u\in U_{ad}$. Thus, $F$ is $\preccurlyeq$-lower semicontinuous on $U_{ad}$ due to Proposition \ref{p2.3}. Let $\varepsilon>0$ be given. The hypotheses of Corollary \ref{co3.1} are then satisfied, so there exists a control $u_\varepsilon\in U_{ad}$ such that, for any $u\in U_{ad}\backslash\{u_\varepsilon\}$, $F(u)+[\varepsilon\|u-u_\varepsilon\|_{U_{ad}},\varepsilon\|u-u_\varepsilon\|_{U_{ad}}]\not\preccurlyeq F(u_\varepsilon)$. By Remark \ref{re3.2}, we can see that, for any $u\in U_{ad}$, $F(u)+[\varepsilon\|u-u_\varepsilon\|_{U_{ad}},\varepsilon\|u-u_\varepsilon\|_{U_{ad}}]\not\prec F(u_\varepsilon)$ and hence
$u_\varepsilon$ is an $\varepsilon$-minimal solution of (IOCP).
\end{proof}

By Theorem \ref{th4.7} and Lemmas \ref{le4.3}-\ref{le4.8}, using the similar argument as in Theorem \ref{th4.8}, we can show the following theorem.
\begin{theorem}\label{th4.9}
Let the (ii)-differentiability be taken on equation (\ref{eq4.14}). Assume that (A1)-(A4) are satisfied. For each $u\in U_{ad}$ and $t\in [0,T]$, if the Hukuhara difference $\chi\ominus_H\int_0^t -f(s,y(s),u(s))ds$ exists for all $y\in C([0,T],\mathcal{I})$ and $\chi\in \mathcal{I}$, then there exists an $\varepsilon$-minimal solution of (IOCP) for every given $\varepsilon>0$.
\end{theorem}

\section{Conclusions}
\label{sec6}
\setcounter{equation}{0}
The present paper focuses on the study of Ekeland's variational principle for interval-valued functions with applications.
The main contributions of this paper are summarized as follows: (i) some new interval-valued versions of Ekeland's variational principle are established under mild conditions; (ii) some new fixed point theorems are obtained for set-valued mappings and interval-valued functions;  (iii) some existence theorems concerned with minimal solutions are shown for interval-valued optimization problems involving the Palais-Smale condition; (iv) a new version of Mountain Pass Theorem is presented for an interval-valued function; (v) an existence theorem of $\varepsilon$-Nash equilibria is given for a noncooperative interval-valued game; (vi) some existence theorems concerned with $\varepsilon$-minimal solutions are deduced for interval-valued optimal control problems driven by interval-valued differential equations under the generalized differentiability.

It is well known that Ekeland's variational principle is closely related to partial differential equations and critical points theory. Naturally,
future research will adopt the presented results in the paper to study interval-valued partial differential equations and critical points theory.


\end{document}